\documentclass{article}
\usepackage{amsmath}
\usepackage{latexsym,amssymb,dsfont}
\usepackage{thmtools,mathtools,amsthm}

\usepackage[a4paper, total={6in, 8in}]{geometry}

\usepackage{stmaryrd}
\usepackage{tikz-cd}
\usepackage{quiver}
\usepackage[numbers]{natbib}
\usepackage{graphicx}
\usepackage{macros}
\usepackage{hyperref}
\usepackage{xcolor}
\usepackage{todonotes}
\usepackage{url}
\usepackage{fancyhdr}

\usepackage{comment}

\usepackage{bussproofs}

\usepackage[capitalize,noabbrev]{cleveref}

\usepackage[capitalize,noabbrev]{cleveref}

\theoremstyle{plain} 
\newtheorem{theorem}{Theorem}[section]
\newtheorem{cor}[theorem]{Corollary}

\newtheorem{proposition}[theorem]{Proposition}
\newtheorem{corollary}[theorem]{Corollary}

\usepackage{cleveref}
\theoremstyle{definition} 
\newtheorem{definition}[theorem]{Definition}
\newtheorem{remark}[theorem]{Remark}
\newtheorem{example}[theorem]{Example}

\title{Skolem, Gödel, and Hilbert fibrations}
\author{Davide Trotta
\footnote{
University of Padova,
\href{mailto:trottadavide92@gmail.com}{trottadavide92@gmail.com}
}
\and
Jonathan Weinberger
\footnote{
Department of Mathematics, 
Johns Hopkins University, 
\href{mailto:jweinb20@jhu.edu}{jweinb20@jhu.edu}
} \and Valeria de Paiva
\footnote{
Topos Institute,
\href{mailto:valeria@topos.institute}{valeria@topos.institute}
}
}

\date{\today}

\begin{document}

\maketitle

\begin{abstract}
    Grothendieck fibrations are fundamental in capturing the concept of \emph{dependency}, notably in categorical semantics of type theory and programming languages. A relevant instance are \emph{Dialectica fibrations} which generalise Gödel's Dialectica proof interpretation and have been widely studied in recent years.  

We characterise when a given fibration is a generalised, \emph{dependent} Dialectica fibration, 
namely an iterated completion of a fibration by \emph{dependent} products and sums (along a given class of \emph{display maps}). 
From a technical perspective, we
complement the work of Hofstra on Dialectica fibrations by an internal viewpoint, categorifying the classical notion of quantifier-freeness. We also generalise both Hofstra's and Trotta et al.'s work on G\"odel fibrations
to the dependent case, replacing the class of cartesian projections in the base category by arbitrary display maps. 
We discuss how
this recovers a range of relevant examples in categorical logic and proof theory. Moreover, as another instance, we introduce \emph{Hilbert fibrations}, providing a categorical understanding of Hilbert's $\epsilon$- and $\tau$-operators well-known from proof theory.
\end{abstract}


\renewcommand{\headrulewidth}{0pt}

\pagestyle{fancy}
\fancyhead{} 
\fancyhead[CO,CE]{\textsc{D.~Trotta, J.~Weinberger, V.~de~Paiva: Skolem, Gödel, and Hilbert fibrations}}
\fancyfoot{} 
\fancyfoot[CE,CO]{\thepage}

\section{Introduction}
Gödel's \emph{Dialectica interpretation} (1958) aimed to reduce the problem of proving
the consistency of first-order arithmetic to the problem of proving the consistency of a
simply-typed system of computable functionals, the well-known \emph{System T}~\cite{godel58,goedel1986}.  Thirty years later, de Paiva introduced a categorification of Gödel's construction~\cite{depaiva1991dialectica}, by assigning to (a finitely complete) category $\mathsf C$ its \emph{Dialectica category} $\mathsf{Dial}(\mathsf C)$. In the following years, several people continued the study of the Dialectica interpretation from a categorical perspective. In particular, work of Hyland~\cite{Hyland2002}, Biering~\cite{Biering_dialecticainterpretations}, Hofstra~\cite{hofstra2011}, von Glehn, and Moss~\cite{moss2018dialectica}, generalised the Dialectica construction, assigning to a Grothendieck fibration $\mathsf p: \mathsf E \to \mathsf B$ its \emph{Dialectica fibration} $\mathfrak{Dial}(\mathsf p)$.
These works, particularly Hofstra's paper \cite{hofstra2011}, highlighted an abstract property underlying the Dialectica interpretation, namely the universal property of \emph{being an} $\exists\forall$\emph{-completion}.
The study of such free constructions involving quantifiers has played a significant role in the investigation of realizability in categorical logic~\cite{hofstra_2006}. Hofstra's result concerning the Dialectica interpretation aligns with this line of research.

In the past decade, the study and the application of these free (quantifier-like) completions have been addressed in various fields and by several authors: the first author introduced a general notion of $\exists$-completion in~\cite{trotta_ex_comp} in the context of Lawvere doctrines and proved that this construction is lax idempotent. Then, in joint work with Maietti, they provided an intrinsic characterisation of the $\exists$-completion~\cite{maiettitrotta2023} and used this construction to characterise the exact completion of elementary and existential doctrines~\cite{maiettitrotta24}. At the same time, and independently, Frey also provided an intrinsic description of the $\exists$-completion in the categorical setting of (posetal) fibrations, in order to use this tool to further investigate realizability from a categorical perspective~\cite{Frey2014AFS,Frey2020}. In the same setting, Maschio and Trotta used this notion to introduce and characterise a general notion of \emph{category of assemblies}~\cite{Trotta2023APAL}.
 We can also find applications of the $\exists$-completion in topos theory, through the notion of the geometric completion~\cite{wrigley2023}.

Concerning specific applications to the Dialectica interpretation, in recent work~\cite{trotta_et_al:LIPIcs.MFCS.2021.87}, Trotta et al.~generalised to the fibrational setting the characterisation presented in~\cite{maiettitrotta2023}, proving an \emph{internal} characterisation of the Dialectica construction, introducing  \emph{Skolem} and \emph{Gödel fibrations} 
as well as \emph{quantifier-free elements} of a fibration.  A relevant application of these notions and results is that they allow us to prove (in the proof-irrelevant setting) that the Dialectica doctrines satisfy the logical principles involved in Gödel's translation~\cite{trotta-lfcs2022,trotta23TCS,dialecticaprinciples2022}.

The main purpose of this work is to generalise the results presented in~\cite{trotta_et_al:LIPIcs.MFCS.2021.87} to the \emph{dependent case}, namely we are interested in characterising the constructions that freely add products and sums (or coproducts) to a fibration along an arbitrary class of display maps on the base.

The main reason we are interested in such a generalisation is that, while it is quite rare to find non-syntactical and genuine examples of fibrations arising as instances of the simple products and/or simple coproduct completions, we realized that the main used fibrations in the literature (such as the subobject or the codomain fibration) arise as instances of the dependent version of these completions. Moreover,  notions of \emph{polynomials} can also be captured through these dependent versions of the completions.

From a conceptual point of view, these generalisations provide us with a useful formal tool to properly compare and highlight the underlying common structure and the differences of categories of fibrations which, over the years, have been noticed to be similar to each other, such as Dialectica categories and categories of polynomials~\cite{moss2022:polyTalk}.

From a technical point of view, the \emph{external} generalisation of the Dialectica construction works as follows: recall that the objects of a Dialectica fibration are given by tuples $(I, U, X,\alpha)$ where $I, U, X$ are objects in the base and $\alpha$ is an element in the fibre over their product, playing the role of a predicate $\exists u\forall x\alpha(i,u,x)$. To extend the previous setting  to dependent types, we replace the  objects of a Dialectica fibration,  by tuples 
$\big(I, U \to I, \sum_u X_u,\alpha(i,u,x)\big)$ where $\alpha$ is an object over the fibre of $X$. 
In the non-dependent case, the completion process is done by adding products, then sums \emph{with respect to the class of cartesian projections} $\{I \times U \to I\}_{I, U \in \mathsf C}$. 
When we generalise  to  the \emph{dependent case}
we replace cartesian projections by maps of a fixed class of display maps $\mathcal F$ and,  (cartesian) exponents by $\mathcal F$-\emph{dependent products}. Thus the (simple) Dialectica fibration of a fibration $\mathsf p$ gets replaced by its generalised variant $\mathfrak{Dial}_\mathcal F(\mathsf p)$, which arises by freely adding fibred products and sums along, more generally, display maps in $\mathcal F$.

In order to provide an internal characterisation of this construction, following along the same lines as \cite{trotta_et_al:LIPIcs.MFCS.2021.87}, we start by generalising the crucial notions of $\coprod$-quantifier and $\prod$-quantifier-free elements (on which the main characterisation presented in \cite{trotta_et_al:LIPIcs.MFCS.2021.87} is built) as well as the notion of Skolem and Gödel fibrations.

The main intuition is that (dependent) Skolem fibrations are fibrations where every element can be written as a $\coprod_{\cF}$ (or $\sum_{\cF}$) of a $\coprod_{\cF}$-quantifier-free element satisfying a form of Skolemisation (relative to the class of display maps $\cF$). Gödel fibrations are Skolem fibrations where every element can be written as a $\coprod_{\cF} \prod_{\cF}$ (or  $\sum_{\cF} \prod_{\cF}$) of an element that is $\coprod_{\cF}$-quantifier-free and $\prod_{\cF}$-quantifier-free (with respect to the full subfibration of $\coprod_{\cF}$-quantifier-free elements).

Notice that, while in the simple product-coproduct case we only need to require exponents in the base to properly state the principle of Skolemisation, moving to the dependent case requires identifying another suitable diagram in the base along with considering the dependent Skolemisation. To achieve this problem, we identify the notion of \emph{strong dependent products}, that are a strong version of the notion of weak dependent products considered by Carboni, Rosolini~\cite{CarboniRosolini2000} and Menni~\cite{MENNI2002}.

Moreover, in this work, we take further advantage of these notions, and we introduce the novel notion of fibrations with Hilbert $\epsilon$- and $\tau$-operators. As we will prove, these notions provide  a categorification in the setting of fibrations of Hilbert's $\epsilon$- and $\tau$-calculus \cite{bell,Devidi95}. Recall that a categorification of the notion of Hilbert's $\epsilon$-operator in the language of Lawvere doctrines has already appeared in the work  \cite{maiettipasqualirosolini} by Maietti, Pasquali, and Rosolini, and it has been further studied and used in \cite{maiettitrotta2023,trotta23TCS}. However, its generalisation to the proof-relevant setting is non-straightforward, and it requires  proper categorical notions of quantifier-free elements for fibrations to be properly addressed.

\section{Grothendieck fibrations}
In this section, we briefly recall some standard notions regarding fibrations. We borrow heavily from \cite{Jacobs1999} and \cite{taylor199} for our presentation of these definitions.

\begin{definition}[cartesian arrow]
Let $\fibration{\mE}{\mp}{\mB}$ be a functor and $\arrow{X}{f}{Y}$ an arrow in $\mE$. Let us call $\arrow{A}{u:=\mp (f)}{B}$ the arrow $\mp(f)$ of $\mB$. We say that $f$ is \textbf{cartesian over} $\boldsymbol{u}$ if, for every morphism $\arrow{Z}{g}{Y}$ in $\mE$ and every morphism $\arrow{C}{w}{A}$ in $\mB$ with $\mp (g)=u w$ there exists a unique arrow $\arrow{Z}{h}{X}$ of $\mE$ such that $g=f h$ and $\mp (h)=w$.
\end{definition}

\begin{figure}
    \centering
\[\begin{tikzcd}
	&& Z \\
	{\mathsf{E}} &&& X && Y \\
	\\
	&& C \\
	{\mathsf{B}} &&& A && B
	\arrow["u"', from=5-4, to=5-6]
	\arrow["{\mathsf{p}(g)=uw}", from=4-3, to=5-6]
	\arrow["{\mathsf{p}}"', from=2-1, to=5-1]
	\arrow["{\forall \,g}", from=1-3, to=2-6]
	\arrow["f"', from=2-4, to=2-6]
	\arrow["{\exists !\, h}"', dashed, from=1-3, to=2-4]
	\arrow["{\forall\, w}"', from=4-3, to=5-4]
\end{tikzcd}\]
    \caption{Universal property of cartesian arrows}
    \label{fig:cart-arr}
\end{figure}
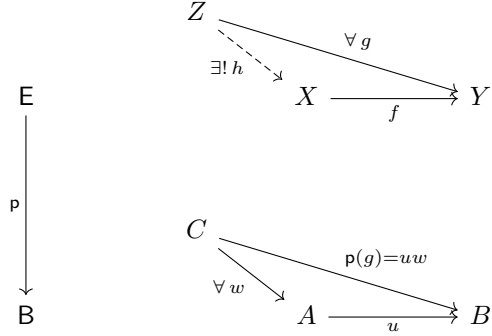

\begin{definition}[Grothendieck fibration]
A \bemph{(Grothendieck) fibration} is a functor $\fibration{\mE}{\mp}{\mB}$ such that, for every $Y$ in $\mE$ and every $\arrow{I}{u}{\mp Y}$, there exists a cartesian arrow  $\arrow{X}{f}{Y}$ over $u$. Such an arrow $f$ is called a \textbf{cartesian lifting} of $u$ with respect to $Y$.
\end{definition}

For a given fibration $\fibration{\mE}{\mp}{\mB}$, and for any $A$ in $\mB$, let $\mE_A$ be the \textbf{fibre category} over $A$: its objects are the objects $X$ of $\mE$ such that $\mp X=A$, and its morphisms, which are said to be \textbf{vertical}, are the morphisms $\arrow{X}{f}{Y}$ of $\mE$ such that $\mp (f)=\id_A$.

It is well-known that cartesian liftings of an arrow with respect to a fixed codomain are determined uniquely up to unique vertical isomorphism. A choice of cartesian liftings is called a \textbf{cleavage}.

Recall from \cite{Jacobs1999} that a fibration is called \textbf{cloven} if it comes together with a choice of cartesian listings, \ie~a cleavage; and it is called \textbf{split} if it is cloven and the given liftings
are well-behaved in the sense that they satisfy certain strict functoriality conditions.

By the global axiom of choice, any fibration can always be assumed cloven (but not necessarily split).

In a cloven fibration $\fibration{\mE}{\mp}{\mB}$, for every morphism $\arrow{B}{u}{\mp(Y)}$ of $\mB$ we denote the chosen cartesian lifting of $u$ by $\arrow{u^\ast (Y)}{\ovln{u}Y}{Y}$. Then we can define the \textbf{substitution functor}: $$ \arrow{\mE_B}{u^\ast}{\mE_A}$$ sending $X$ to $u^\ast(X)$ and a vertical morphism $\arrow{X}{f}{Y}$ to the unique mediating map $u^*(f)$ in:
\[\begin{tikzcd}
	{u^*(X)} & X \\
	{u^*(Y)} & Y
	\arrow["{\overline{u}(Y)}"', from=2-1, to=2-2]
	\arrow["f", from=1-2, to=2-2]
	\arrow["{\overline{u}X}", from=1-1, to=1-2]
	\arrow["{u^*(f)}"', dashed, from=1-1, to=2-1]
\end{tikzcd}\]
The next concept we need to recall is the \textbf{opposite} of a fibration. Recall from \cite[Lemma 1.4.10]{Jacobs1999} that, given a fibration $\fibration{\mE}{\mp}{\mB}$, for every cleavage of $\mp$ one has the isomorphism of sets (or classes): 
$$ \begin{aligned}
\mE(X,Y)&\overset{\cong}{\to} \coprod_{u:\mp X \rightarrow \mp Y} \mE_{\mp X}(X, u^{\ast}(Y)) \\
f&\mapsto (\mp f, f')
\end{aligned}$$
where $\coprod$ is the disjoint union and $f'$ is the unique vertical arrow such that $f=(\overline{ \mp f}(\mp Y))f'$. This means that a morphism in a total category $\mE$ corresponds to a morphism in the basis together with a vertical map. The intuition behind the definition of the opposite fibration is that all vertical maps in such composites are reversed. 

\begin{definition}[opposite fibration, \cf~\cite{Jacobs1999}]\label{def_op_fibration}
Let $\fibration{\mE}{\mp}{\mB}$ be a fibration.  We describe a new fibration over the same base written as $\opfibration{\mE}{\mp}{\mB}$, which is fibrewise the opposite of $\mp$, called the \textbf{(fibrewise) opposite} or \textbf{dual} of $\mp$.
\end{definition}
Let $CV$ be the class:
$$\{ (f_1,f_2)| f_1\mbox{ is Cartesian, } f_2 \mbox{ is vertical and} \dom (f_1)=\dom (f_2)\}.$$ An equivalence relation is defined on the collection $CV$ by $(f_1,f_2)\sim (g_1,g_2)$ if there exists an arrow $h$ such that $f_1=g_1h$ and $f_2=g_2h$. The equivalence class of $(f_1,f_2)$ is denoted by $[f_1,f_2]$.
The total category $\mE^{(\op)}$ has the same objects of $\mE$, and morphisms $X\rightarrow Y$ are equivalence classes $[f_1,f_2]$ of arrows $f_1$ and $f_2$, with $f_1$ cartesian and $f_2$ vertical, as in:
\[\begin{tikzcd}
	X \\
	\bullet & Y
	\arrow["{f_2}", from=2-1, to=1-1]
	\arrow["{f_1}"', from=2-1, to=2-2]
\end{tikzcd}\]
The composition $[g_1,g_2]\circ [f_1,f_2]$ is described by the following diagram:
\[\begin{tikzcd}[column sep=huge, row sep=huge]
	X \\
	A & Y \\
	{\mp f_1^*(B)} & B & Z
	\arrow["{f_2}", from=2-1, to=1-1]
	\arrow["{f_1}"', from=2-1, to=2-2]
	\arrow["{g_2}"', from=3-2, to=2-2]
	\arrow["{g_1}"', from=3-2, to=3-3]
	\arrow["{\overline{\mp f_1}B}"', dashed, from=3-1, to=3-2]
	\arrow["x", dashed, from=3-1, to=2-1]
\end{tikzcd}\]
where $x$ is the unique vertical arrow arising by cartesianness
of $f_1$ and making the diagram commute. We define the composition $[g_1,g_2]\circ [f_1,f_2]$ to be the class: $$[g_1(\ovln{\mp f_1} B),f_2x]$$ which turns out to be well-defined. See \cite[Definition 1.10.11]{Jacobs1999} for more details.

The functor $\opfibration{\mE}{\mp}{\mB}$ is defined by the assignments $X\mapsto \mp X$ and $[f_1,f_2]\mapsto \mp (f_1)$, and it is well-defined since $f_2$ is vertical.

We recall some well-known examples of fibrations:
\begin{example}[codomain fibration]\label{ex:codomain fibration}
    For an arbitrary category $\mB$ we denote by $\mB^{\to}$ its arrow category: the objects of $\mB^{\to}$ are arrows of $\mB$, and an arrow from $X\xrightarrow{f}Y$ to $K\xrightarrow{g}Z$ is given by a pair of morphisms $\arrow{X}{u}{K}$ and $\arrow{Y}{v}{Z}$ such that $ vf=gu$. The codomain functor $\fibration{\mB^{\to}}{\mathsf{cod}}{\mB}$ is a fibration exactly when $\mB$ has pullbacks (cartesian morphisms coincide with pullbacks). The fibre category over an object $X$ of $\mB$ is given by the slice category $\mB/X$. The opposite of the codomain fibration is the category of \emph{lenses}, \cf~\cite{SpivakGenLenses,CapucciMRC2024}.
\end{example}
\begin{example}[family fibration of a category]\label{ex:family fibration}
    For an arbitrary category $\mB$, we denote by $\Fam (\mB)$ the category of set-indexed families of objects and arrows of $\mB$: objects of $\Fam (\mB)$ are collections $(X_i)_{i\in I}$ of objects $X_i$ of $\mB$ such that every $i$ is an element of a set $I$. An arrow $(X_i)_{i\in I}\to (Y_j)_{j\in J}$ of $\Fam (\mB)$ is given by a function $\arrow{I}{u}{J}$ and a family $(f_i)_{i\in I}$ of morphisms $\arrow{X_i}{f_i}{Y_{u(i)}}$ in $\mB$.  The projection functor $\fibration{\Fam (\mB)}{\mp}{\set}$ mapping $(X_i)_{i\in I}\to I$ and $(u,(f_i)_{i\in  I})\to u$ is a fibration called \textbf{family fibration}. The fibre category $\Fam (\mB)_I$ over a set $I$ is the category of $I$-indexed families of objects and arrows in $\mB$.
\end{example}

\subsection{Opfibrations and bifibrations}

While fibrations admit contravariant transport between the fibres, there exists a dual notion of fibration whose transport is covariant.

\begin{definition}[cocartesian arrows and opfibrations]
    Let $\fibration{\mE}{\mp}{\mB}$ be a fibration. A \textbf{cocartesian arrow} in $\mp$ is a cartesian arrow in $\fibration{\mE^{\op}}{\mp^{\op}}{\mB^{\op}}$. We call $\mp$ a \textbf{(Grothendieck) opfibration} if $\mp^{\op}$ is a fibration.
\end{definition}

If $\fibration{\mE}{\mp}{\mB}$ is an opfibration and $\arrow{A}{u}{B}$ an arrow in $\mB$ then by cocartesian transport we get an induced functor $\coprod_u \colon \mE_A \to \mE_B$. If a functor $\fibration{\mE}{\mp}{\mB}$ is both a fibration and an opfibration we call it a \textbf{bifibration}. If, in addition, for every arrow $\arrow{A}{u}{B}$ cartesian reindexing $u^*$ posesses a right adjoint, written $\prod_u$, we call $\fibration{\mE}{\mp}{\mB}$ a \textbf{trifibration}. Hence, trifibrations admit adjoint triples
\[ \coprod_u \dashv u^* \dashv \prod_u.\]
In logic, $u^*$ is reindexing or substitution, while $\coprod_u$ can be understand as existential quantification or a dependent sum object former, and $\prod_u$ is universal quantification or a dependent function object former. This fundamental idea of viewing quantifiers as adjoints goes back to Lawvere~\cite{Lawvere:Adjoint}. We will make this more precise in the next subsection. However, we will usually require the left and right adjoints to exist only for a subclass of morphisms of the base category.

\paragraph{Notation.}
We will employ the following notation for the kinds of arrows in (op)fibrations. We denote:
\begin{itemize}
    \item vertical arrows by $\rightsquigarrow$
    \item cartesian arrows by $\cartarr$
    \item cocartesian arrows by $\cocartarr$
\end{itemize}

\subsection{Fibred (co)products}
In categorical logic, the notion of \emph{display map} generalises the ordinary notion of \emph{product projection}, and plays a crucial role in the categorical semantics of dependent type theory. In the following definition, we recall the notion of \emph{display map category} from \cite[Def. 10.4.1]{Jacobs1999} that is, among the various definitions appearing in the literature, the most general as it just requires the closure under pullbacks:
\begin{definition}[display map categories]\label{def:disp-maps}
Let $\mathsf B$ be a category. A class $\cF \subseteq \mor(\mB)$ is said to be a class of \textbf{display maps} if $\cF$ is closed under pullbacks along arbitrary maps in $\mB$, namely pullbacks along arrows of $\cF$ exist and belong to $\cF$. If an arrow $u:K \to I$ is in $\cF$ we write it as $u:K \tofib I$.
A \textbf{display map category} is a pair $\pair{\mB}{\cF}$ where $\mB$ is a category and $\cF$ a class of display maps.
\end{definition}

\begin{definition}[closure properties of display map categories]\label{def_well_rooted_disp_map_cat}
 A display map category $\pair{\mB}{\cF}$:
 \begin{itemize}
    \item has \textbf{units} if all the isomorphism of $\mB$ are in $\cF$;
     \item has $\cF$-\textbf{dependent coproducts}\footnote{A display map category satisfying this condition is said to have \emph{strong sums} in \cite[p. 610]{Jacobs1999}.} if $\cF$ is closed under composition;
     \item is \textbf{well-rooted} if $\mB$ has a terminal object $1$ and for every object $X$ of $\mB$ the unique arrow $\arrow{X}{!_X}{1}$ is in $\cF$.
 \end{itemize}
\end{definition}

Given a display map category $\pair{\mB}{\cF}$, we will denote by $\cF^{\to}$ the full subcategory of the arrow category $\mB^{\to}$ whose objects are arrows of $\cF$, and we denote by $\fibration{\cF^{\to}}{\mathsf{cod}}{\mB}$ the corresponding (full) subfibration of the codomain fibration $\fibration{\mB^{\to}}{\mathsf{cod}}{\mB}$ (following the notation used in  \cite[p. 610]{Jacobs1999}.

Now we consider a ``strong'' and ``display  map-relative'' version of the notion \emph{weak dependent products} as presented in \cite[Def. 2.1]{MENNI2002} and \cite{CarboniRosolini2000}. Recall that among the categories with finite limits, those with weak dependent products are exactly the ones whose exact completions are locally cartesian
closed.
\begin{definition}[$\cF$-dependent products]\label{def:F-strong-dep-prod}
    Let $\pair{\mB}{\cF}$ be a display map category. An $\cF$-\textbf{dependent product}  of an arrow $f:K \tofib I$ of $\cF$ along another arrow $g:I \tofib J$  of $\cF$ consists of a commutative diagram 
\[\begin{tikzcd}[column sep=huge, row sep=huge]
	K & Z & E \\
	& I & J
	\arrow["f"', two heads, from=1-1, to=2-2]
	\arrow["g"', two heads, from=2-2, to=2-3]
	\arrow[two heads, from=1-2, to=2-2]
	\arrow["h", two heads, from=1-3, to=2-3]
	\arrow[two heads, from=1-2, to=1-3]
	\arrow["e"', from=1-2, to=1-1]
	\arrow["\lrcorner"{anchor=center, pos=0.125}, draw=none, from=1-2, to=2-3]
\end{tikzcd}\]
where the square is a pullback and $h: E\tofib J$ is an arrow of $\cF$, such that for every commutative diagram
\[\begin{tikzcd}[column sep=huge, row sep=huge]
	K & Z' & E' \\
	& I & J
	\arrow["f"', two heads, from=1-1, to=2-2]
	\arrow["g"', two heads, from=2-2, to=2-3]
	\arrow[two heads, from=1-2, to=2-2]
	\arrow["h'", two heads, from=1-3, to=2-3]
	\arrow[two heads, from=1-2, to=1-3]
	\arrow["e'"', from=1-2, to=1-1]
	\arrow["\lrcorner"{anchor=center, pos=0.125}, draw=none, from=1-2, to=2-3]
\end{tikzcd}\]
there exists a unique pair of arrows $w:Z'\to Z$ and $k:E'\to E$ (neither of them necessarily in $\cF$) such that the diagram
\[\begin{tikzcd}[column sep=huge, row sep=huge]
	&& {Z'} & {E'} \\
	K & Z & E \\
	& I & J
	\arrow["f"', two heads, from=2-1, to=3-2]
	\arrow["g"', two heads, from=3-2, to=3-3]
	\arrow[two heads, from=2-2, to=3-2]
	\arrow["h", two heads, from=2-3, to=3-3]
	\arrow["e"', from=2-2, to=2-1]
	\arrow["\lrcorner"{anchor=center, pos=0.125}, draw=none, from=2-2, to=3-3]
	\arrow["{h'}", curve={height=-12pt}, two heads, from=1-4, to=3-3]
	\arrow[two heads, curve={height=-6pt}, from=1-3, to=3-2]
	\arrow[two heads, from=1-3, to=1-4]
	\arrow["{e'}"', curve={height=12pt}, from=1-3, to=2-1]
	\arrow["w"', dashed, from=1-3, to=2-2]
	\arrow["k", dashed, from=1-4, to=2-3]
	\arrow["\lrcorner"{anchor=center, pos=0.125}, draw=none, from=1-3, to=3-3]
    \arrow[two heads, from=2-2, to=2-3, crossing over]
\end{tikzcd}\]
commutes. We say that  $\pair{\mB}{\cF}$ has $\cF$-\textbf{dependent products} if any pair of arrows of $\cF$ has an $\cF$-dependent product.
\end{definition}
The ``strong version'' of the original definition of weak dependent products is obtained as a particular case of the previous one by considering the class of all the morphisms of a given category $\mB$ with finite limits. The use of the word ``weak'' in the original setting is motivated by the fact that the uniqueness of the stipulated arrow $k$ in \Cref{def:F-strong-dep-prod} is not required.

\begin{corollary}
Note that the existence of an arrow $k:E'\to E$ in~\Cref{def:F-strong-dep-prod} implies, by the standard properties of pullbacks, that the arrow $w:Z'\to Z$ is $w=(h^*g)^*k$. Hence, the condition of~\Cref{def:F-strong-dep-prod} is tantamount to just demanding the existence of a map $k$ as indicated such that $h' = hk$.
Furthermore, note that the maps $g^*h$ and $g^*h'$ both are in $\cF$ as well by pullback closure.
\end{corollary}

\begin{example}\label{ex_dep_prod_and_cart_closed}
Let $\pair{\mB}{\cF}$ be a display map category, where $\mB$ is a cartesian category and $\cF$ is the class of product projections. Then, if the category $\mB$ has $\cF$-dependent products in the sense of \Cref{def:F-strong-dep-prod} it is cartesian closed. In particular we can define an exponent $Y^X$ and the evaluation map $\arrow{X\times Y^X}{\ev}{Y}$ by considering the $\cF$-dependent product of the product projection $\arrow{X\times Y}{\pr_X}{Y}$ along the terminal projection $\arrow{X}{!_X}{1}$:

\[\begin{tikzcd}[column sep=huge, row sep=huge]
	X\times Y & X\times Y^X & Y^X \\
	& X & 1
	\arrow["\pi_X"', two heads, from=1-1, to=2-2]
	\arrow["!_X"', two heads, from=2-2, to=2-3]
	\arrow["\pi_X",from=1-2, to=2-2]
	\arrow["!_{Y^X}", two heads, from=1-3, to=2-3]
	\arrow[two heads, "\pi_{Y^X}",from=1-2, to=1-3]
	\arrow["\angbr{\pr_X}{\ev}"', from=1-2, to=1-1]
	\arrow["\lrcorner"{anchor=center, pos=0.125}, draw=none, from=1-2, to=2-3]
\end{tikzcd}\]
\end{example}

\begin{remark}
   Observe that when $\cF$ is the class of all the morphisms of $\mB$ the previous example can be generalised to show that $\mB$ is locally cartesian closed. For more details about some variants of the notion of dependent products and their link with the notion of exact completions and locally cartesian closed category, we refer to \cite{MENNI2002} and \cite{EMMENEGGER2020106414}.
\end{remark}
\begin{definition}[Fibrations with fibred (co)products]\label{def:fib-compl}
Let $\pair{\mB}{\cF}$ be a display map category. A Grothendieck fibration $\fibration{\mE}{\mP}{\mB}$ is said to \textbf{have (fibred) coproducts along $\cF$} or \textbf{$\cF$-coproducts} whenever the following conditions are satisfied:
\begin{enumerate}
    \item for any $v:L \tofib J$ in $\cF$, the reindexing functor $v^*: \mE_J \to \mE_L$ has a left adjoint $\coprod_v: \mE_L \to \mE_J$.
    \item for each pullback of the form
\[\begin{tikzcd}
	K && I \\
	L && J
	\arrow["g"', from=1-1, to=2-1]
	\arrow["u", two heads, from=1-1, to=1-3]
	\arrow["v"', two heads, from=2-1, to=2-3]
	\arrow["f", from=1-3, to=2-3]
	\arrow["\lrcorner"{anchor=center, pos=0.125}, draw=none, from=1-1, to=2-3]
\end{tikzcd}\]
in $\mathsf B$, the following natural transformation is an isomorphism, \ie, the \textbf{Beck--Chevalley condition} holds:
\begin{align}\label{eq:left-bcc}
    \coprod_u g^* \stackrel{\cong}{\Longrightarrow} f^*\coprod_v
\end{align}
Analogously, $\mathsf p$ is said to have \textbf{have (fibred) products along $\cF$} or \textbf{$\cF$-products} if for any $v \in \cF$ the reindexing functor $v^*$ has a right adjoint $\prod_v$, and for any square as above the following Beck--Chevalley condition for the adjoint pair $v^* \dashv \prod_v$ is satisfied, \ie, the following natural transformation is an isomorphism:
\begin{align}\label{eq:right-bcc}
f^* \prod_v \stackrel{\cong}{\Longrightarrow} \prod_u g^*
\end{align}

Fixing $v \in \cF$, if (\ref{eq:left-bcc}) holds for all $f$ as indicated, we say that $v$ \textbf{satisfies the left BCC}. Analogously, we say that $v$ \textbf{satisfies the right BCC} if (\ref{eq:right-bcc}) holds for all $f$ as indicated.\footnote{Maps for which reindexing has a left adjoint satisfying the left BCC are often called \emph{smooth} in geometric contexts. Maps for which reindexing has a right adjoint satisfying the right BCC are called \emph{proper}. A unifying view of smooth and proper maps in geometry and logic is given in~\cite{AnelWeinberger:2024}.}

\end{enumerate}
\end{definition}

The following proposition explains how having $\cF$-dependent (co)products for a display map category $\pair{\mB}{\cF}$ provides a stronger property than having fibred $\cF$-(co)products for $\fibration{\cF^{\to}}{\mathsf{cod}}{\mB}$.
\begin{proposition}
Let $\pair{\mB}{\cF}$ be a display map category, and let $\fibration{\cF^{\to}}{\mathsf{cod}}{\mB}$ be the codomain fibration.  Then 
\begin{enumerate}
    \item $\mB$ has $\cF$-dependent coproducts if and only if the codomain fibration $\fibration{\cF^{\to}}{\mathsf{cod}}{\mB}$ has fibred $\cF$-coproducts;
    \item if $\mB$ has $\cF$-dependent products then the codomain fibration $\fibration{\cF^{\to}}{\mathsf{cod}}{\mB}$ has fibred $\cF$-products.
\end{enumerate}

\end{proposition}
\begin{proof}
1) This is a standard result of categorical logic. We refer to \cite{Jacobs1999} or \cite{taylor199} for all the details.

    2) Suppose that $\mB$ has $\cF$-dependent products. Then for any arrow $f:I \tofib J$ of $\cF$, we can define a functor $\prod_f: \cF_I^{\to} \to \cF_J^{\to}$, where $\prod_f(v) \colon E \fibarr J$ is part of the $\cF$-dependent product diagram 
    \[\begin{tikzcd}[column sep=huge, row sep=huge]
	K & Z & E \\
	& I & J
	\arrow["v"', two heads, from=1-1, to=2-2]
	\arrow["f"', two heads, from=2-2, to=2-3]
	\arrow[two heads, from=1-2, to=2-2]
	\arrow["\prod_f(v)", two heads, from=1-3, to=2-3]
	\arrow[two heads, from=1-2, to=1-3]
	\arrow["e"', from=1-2, to=1-1]
	\arrow["\lrcorner"{anchor=center, pos=0.125}, draw=none, from=1-2, to=2-3]
\end{tikzcd}\]
and the action of $\prod_f$ on a morphism $\alpha :u \to v$ in $\cF_I^{\to} $ is defined by employing the universal property of $\cF$-dependent products:
\[\begin{tikzcd}[column sep=huge, row sep=huge]
	&& K' & Z' & {E'} \\
	K & Z & E & {}\\
	& I & J
	\arrow["v"', two heads, from=2-1, to=3-2]
	\arrow["f"', two heads, from=3-2, to=3-3]
	\arrow[two heads, from=2-2, to=3-2]
	\arrow[from=2-2, to=2-1]
	\arrow[curve={height=-18pt}, two heads, from=1-4, to=3-2]
	\arrow["\alpha"', from=1-3, to=2-1]
	\arrow["u"'{pos=0.205}, two heads, from=1-3, to=3-2,]
	\arrow[two heads, from=1-4, to=1-5]
	\arrow[from=1-4, to=1-3]
    \arrow[two heads, from=2-2, to=2-3,crossing over]
	\arrow["{\prod_f(\alpha)}", dashed, from=1-5, to=2-3,crossing over]
	\arrow[dashed, from=1-4, to=2-2, crossing over]
 \arrow["{\prod_f(v)}", two heads, from=2-3, to=3-3,crossing over]
 \arrow["{\prod_f(u)}", curve={height=-18pt}, two heads, from=1-5, to=3-3,crossing over]
 \arrow["\lrcorner"{anchor=center, pos=0.125}, draw=none, from=2-2, to=3-3]
	\arrow["\lrcorner"{anchor=center, pos=0.125}, draw=none, from=1-4, to=2-4]
\end{tikzcd}\]
These assignments provide a right adjoint to the re-indexing $f^*$, \ie
\[ \cF_J^{\to}(k,\prod_f(v))\cong \cF_I^{\to}(f^*k,v) \]
because of the universal property of $\cF$-dependent products and the fact that $f^*$ acts as a pullback for codomain fibrations. Now we show that these right adjoints satisfy the BCC: let us consider the following pullback
	
	\[\begin{tikzcd}[column sep=large, row sep=large]
		Y & X \\
		I & J.
		\arrow["g", from=1-2, to=2-2]
		\arrow["i"', from=1-1, to=2-1]
		\arrow["f"', two heads, from=2-1, to=2-2]
		\arrow["h", two heads, from=1-1, to=1-2]
		\arrow["\scalebox{1.6}{$\lrcorner$}"{anchor=center, pos=0.1}, shift left=3, draw=none, from=1-1, to=2-1]
	\end{tikzcd}\]
 We have to show that $\prod_{h}i^*(v)\cong g^*\prod_f(v)$ for every $v\in \cF_I^{\to}$. Now let us consider the following diagram

\[\begin{tikzcd}[column sep=large, row sep=large]
	&&& {Z'} && {E'} \\
	&& Z && E \\
	& H && Y && X \\
	K && I && J
	\arrow[two heads, from=1-4, to=1-6]
	\arrow[from=1-4, to=2-3]
	\arrow["{{e'}}"', curve={height=30pt}, dashed, from=1-4, to=3-2]
	\arrow[two heads, from=1-4, to=3-4]
	\arrow["\lrcorner"{anchor=center, pos=0.125}, draw=none, from=1-4, to=3-6]
	\arrow[from=1-6, to=2-5]
	\arrow[""{name=0, anchor=center, inner sep=0}, "{{g^*\prod_f(v)}}", two heads, from=1-6, to=3-6]
	\arrow["{{i^*(v)}}"{pos=0.3}, two heads, from=3-2, to=3-4]
	\arrow[from=3-2, to=4-1]
	\arrow["\lrcorner"{anchor=center, pos=0.125}, draw=none, from=3-2, to=4-3]
	\arrow["h"{pos=0.3}, two heads, from=3-4, to=3-6]
	\arrow["i"', from=3-4, to=4-3]
	\arrow["\lrcorner"{anchor=center, pos=0.125}, draw=none, from=3-4, to=4-5]
	\arrow["g", from=3-6, to=4-5]
	\arrow["v"', two heads, from=4-1, to=4-3]
	\arrow["f"', two heads, from=4-3, to=4-5]
	\arrow["{{\text{(I)}}}"', draw=none, from=1-4, to=0]
 \arrow[crossing over, two heads, from=2-3, to=2-5]
	\arrow["e"', crossing over, curve={height=30pt}, from=2-3, to=4-1]
	\arrow[crossing over, two heads, from=2-3, to=4-3]
	\arrow["\lrcorner"{anchor=center, pos=0.125}, draw=none, from=2-3, to=4-5]
	\arrow[""{name=1, anchor=center, inner sep=0}, "{{\prod_f(v)}}"{pos=0.3}, crossing over, two heads, from=2-5, to=4-5]
 	\arrow["\lrcorner"{anchor=center, pos=0.125, rotate=-90}, draw=none, from=1-6, to=1]
\end{tikzcd} \]
First, notice that the square (I) is a pullback by construction, and that we can define the unique arrow $e':Z'\to H$  using the fact that the left square is a pullback.  Combining the universal property of pullbacks with that of $\cF$-dependent products, it is direct to check that

\[\begin{tikzcd}[column sep=large, row sep=large]
	H & {Z'} & {E'} \\
	& X & Y
	\arrow["{i^*(v)}"', two heads,from=1-1, to=2-2]
	\arrow["h"', two heads, from=2-2, to=2-3]
	\arrow["{g^*\prod_f(v)}", two heads, from=1-3, to=2-3]
	\arrow[two heads, from=1-2, to=2-2]
	\arrow["{e'}"', from=1-2, to=1-1]
	\arrow[two heads, from=1-2, to=1-3]
	\arrow["\lrcorner"{anchor=center, pos=0.125}, draw=none, from=1-2, to=2-3]
\end{tikzcd}\]
satisfies the universal property of $\cF$-dependent products, and hence we can conclude that $\prod_h i^*(v)\cong g^*\prod_f(v)$.
 \end{proof}

\section{Fibred (co)product completions}

\subsection{Coproduct completion}

In this section we present a proof-relevant generalisation of the \emph{generalised existential completion} introduced in \cite{maiettitrotta2023}. The crucial idea is that, given a display map category $\pair{\mB}{\cF}$ with $\cF$-dependent coproducts, where $\mB$ has all pullbacks, and a fibration $\fibration{\mE}{\mp}{\mB}$, we can \emph{freely} construct a new fibration denoted by $\fibration{\Sigma_{\cF}(\mE)}{\Sigma_{\cF}(\mp)}{\mB}$ having $\cF$-coproducts. We will call this construction the $\Sigma_{\cF}$\emph{-completion}. 

A particular case of this construction can be found in \cite[Sec. 3.2]{hofstra2011}, where the so-called \emph{family construction} is considered on the level of fibrations. It freely adds left adjoints to reindexing, along \emph{all} the morphisms of a fibration $\fibration{\mE}{\mp}{\mB}$ where $\mB$ is supposed to have finite limits.\\

For the rest of this section, let $\pair{\mB}{\cF}$ be a fixed display map category with $\cF$-dependent coproducts, where $\mB$ has all pullbacks, and let $\fibration{\mE}{\mp}{\mB}$ be a fixed fibration.

\noindent
\textbf{The $\Sigma_{\cF}$-completion.} The category $\Sigma_{\cF}(\mE)$ has:
\begin{itemize}
    \item as \textbf{objects} pairs $(g:X \twoheadrightarrow I, \alpha)$ where $\alpha$ is an object of the fibre $\mE_I$;
    \item as \textbf{morphisms} triples $(f_0,f_1,\phi) : (g:X \twoheadrightarrow I, \alpha)\to (h:Y \twoheadrightarrow J, \beta)$ where 
\[\begin{tikzcd}
	\alpha & \beta \\
	X & Y \\
	I & J
	\arrow["\phi", from=1-1, to=1-2]
	\arrow[Rightarrow, dotted, no head, from=1-1, to=2-1]
	\arrow[Rightarrow, dotted, no head, from=1-2, to=2-2]
	\arrow["{f_1}", from=2-1, to=2-2]
	\arrow["g"', two heads, from=2-1, to=3-1]
	\arrow["h", two heads, from=2-2, to=3-2]
	\arrow["{f_0}"', from=3-1, to=3-2]
\end{tikzcd}\]
is a commutative square in $\mB$ with a map $\phi:\alpha\to \beta$ in $\mE$ over $f_1$.
\end{itemize}
In this case, the idea is that an object $( g:X \twoheadrightarrow I, \alpha)$ is thought of as a predicate of the form $\exists i,x.\alpha(i,x)$

The functor $\fibration{\Sigma_{\cF}(\mE)}{\Sigma_{\cF}(\mp)}{\mB}$ is defined by the assignments
$\Sigma_{\cF}(\mp)( g:I \twoheadrightarrow X, \alpha):= I$ and $\Sigma_{\cF}(\mp)(f_0,f_1,\phi):=f_0$. It is straightforward to check that this functor defines a fibration.

Employing the assumption that the display map category $\pair{\mB}{\cF}$ has $\cF$-dependent coproducts, one can easily check that the fibration $\fibration{\Sigma_{\cF}(\mE)}{\Sigma_{\cF}(\mp)}{\mB}$ has $\cF$-coproducts defined as follows: for every arrow $k:I\to J$, the functor $\coprod_k: \Sigma_{\cF}(\mE)_I\to \Sigma_{\cF}(\mE)_J$ acts as $(g:X \twoheadrightarrow I, \alpha)\mapsto (kg:X \twoheadrightarrow J, \alpha)$ on the objects and as $(f_0,f_1,\phi)\mapsto (\id_J,f_1,\phi)$ on the vertical arrows.

As a construction, the functor $\fibration{\Sigma_{\cF}(\mE)}{\Sigma_{\cF}(\mp)}{\mB}$ can be obtained as follows:
\[\begin{tikzcd}[column sep=large, row sep=large]
	 \Sigma_{\cF}(\mE)  &\mE \\
	\cF^{\to} & \mB \\
	\mB
	\arrow[from=1-1, to=2-1]
	\arrow["\mathsf{cod}",from=2-1, to=3-1]
	\arrow["\mathsf{dom}"',from=2-1, to=2-2]
	\arrow[from=1-1, to=1-2]
	\arrow["\mp",from=1-2, to=2-2]
	\arrow["\Sigma_{\cF}(\mp)"',curve={height=28pt}, from=1-1, to=3-1]
 	\arrow["\lrcorner"{anchor=center, pos=0.125}, draw=none, from=1-1, to=2-2]
\end{tikzcd}\]
From this description it follows from general closure properties that $\Sigma_\cF(\mp)$ is, in fact, a fibration.
By the universal property of the pullback, this construction is easily seen
 to be 2-functorial in a suitable 2-categorical setting, where we consider the 2-category $\mathfrak{DispFib}$ defined as follows:
\begin{itemize}
    \item 0-cells are pairs $(\fibration{\mE}{\mp}{\mB},\pair{\mB}{\cF})$, where $\mp$ is a fibration and
    $\pair{\mB}{\cF}$ is a display map category with $\cF$-dependent coproducts;
    \item 1-cells are commutative diagrams
\[\begin{tikzcd}[column sep=large, row sep=large]
	\mE & \mE' \\
	\mB & \mB'
	\arrow["\mp"',from=1-1, to=2-1]
	\arrow["F"',from=2-1, to=2-2]
	\arrow["\mp'",from=1-2, to=2-2]
	\arrow["F_0",from=1-1, to=1-2]
\end{tikzcd}\]
where $F_0:\mE\to \mE'$ is a \emph{cartesian functor}, \ie~it sends $\mp$-cartesian maps to $\mp'$-cartesian maps, and $F:\mB\to \mB'$ is a functor \emph{preserving display maps}, \ie, $F(f)$ is an arrow in $ \cF'$ for every arrow $f$ of $ \cF$;
\item 2-cells are pairs $(\phi_0,\phi): (F_0,F)\to (G_0,G)$ of natural transformations $\phi_0:F_0\to G_0$ an $\phi:F\to G$ such that the component$(\phi_0)_X$ is sent by $\mp$ to $(\phi)_{\mp'(X)}$ for every object $X$ of $\mE$.
\end{itemize}

A direct generalisation of \cite[Thm.~3.11]{maiettitrotta2023} and \cite[Thm.~3.5]{hofstra2011} gives the following result:
\begin{theorem}\label{thm_coprod_monad}
    The assignment $(\fibration{\mE}{\mp}{\mB},\pair{\mB}{\cF})\mapsto (\fibration{\Sigma_{\cF}(\mE)}{\Sigma_{\cF}(\mp)}{\mB},\pair{\mB}{\cF})$ extends to a 2-monad on the 2-category $\mathfrak{DispFib}$. The 2-category of pseudo-algebras is 2-equivalent to the 2-full sub category of $\mathfrak{DispFib}$ whose objects are pairs $(\fibration{\mE}{\mp}{\mB},\pair{\mB}{\cF})$ where $\mp$ has $\cF$-coproducts and whose 1-cells are coproduct preserving 1-cells of $\mathfrak{DispFib}$.
\end{theorem}

Two relevant examples of fibrations arising as instances of the previous construction are the subobject and the codomain fibrations:
\begin{example}\label{ex_cod_fib_as_coprod_comp}
    Let us consider a category with finite limits $\mB$ and the class of display maps $\cF:=\mor (\mB)$ of all the morphisms of $\mB$ (in this case $\mor (\mB)^{\to}=\mB^{\to}$). Then the codomain fibration $\fibration{\mB^{\to}}{\mathsf{cod}}{\mB}$ can be easily proved to be an instance of the $\Sigma_{\mor (\mB)}$-completion, since it can be readily expressed via the pullback
\[\begin{tikzcd}[column sep=large, row sep=large]
	 \Sigma_{\mor (\mB)}(\mB)  &\mB \\
	\mB^{\to} & \mB \\
	\mB
	\arrow[from=1-1, to=2-1]
	\arrow["\mathsf{cod}",from=2-1, to=3-1]
	\arrow["\mathsf{dom}"',from=2-1, to=2-2]
	\arrow[from=1-1, to=1-2]
	\arrow["\mathsf{id}",from=1-2, to=2-2]
	\arrow["\Sigma_{\mor (\mB)}(\mathsf{id})"',curve={height=28pt}, from=1-1, to=3-1]
 	\arrow["\lrcorner"{anchor=center, pos=0.125}, draw=none, from=1-1, to=2-2]
\end{tikzcd}\]
of the identity fibration.
\end{example}

\begin{example}\label{ex_mono_fib_as_coprod_comp}
        Let us consider a category with finite limits $\mB$ and the class of display maps $\cF:=\mathsf{Mon}(\mB)$ of all the monomorphisms of $\mB$. In this case the category $\mathsf{Mon} (\mB)^{\to}$ is exactly that of subobjects $\mathsf{Sub}(\mB)$. Then the subobject fibration $\fibration{\mathsf{Sub}(\mB)}{\mathsf{sub}}{\mB}$ can be easily proved to be an instance of the $\Sigma_{\mathsf{Mon}(\mB)}$-completion,  since it can be trivially defined via the pullback
\[\begin{tikzcd}[column sep=large, row sep=large]
	 \Sigma_{\mathsf{Mon}(\mB)}(\mB)  &\mB \\
	\mathsf{Sub}(\mB) & \mB \\
	\mB
	\arrow[from=1-1, to=2-1]
	\arrow["\mathsf{cod}",from=2-1, to=3-1]
	\arrow["\mathsf{dom}"',from=2-1, to=2-2]
	\arrow[from=1-1, to=1-2]
	\arrow["\mathsf{id}",from=1-2, to=2-2]
	\arrow["\Sigma_{\mathsf{Mon}(\mB)}(\mathsf{id})"',curve={height=28pt}, from=1-1, to=3-1]
 	\arrow["\lrcorner"{anchor=center, pos=0.125}, draw=none, from=1-1, to=2-2]
\end{tikzcd}\]
of the identity fibration.
\end{example}

\subsection{Product completion}

We conclude this section by presenting the dual construction of the $\Sigma_{\cF}$-completion, namely the $\Pi_{\cF}$\emph{-completion}. Again, a particular case of this construction can be found in \cite[Sec. 3.2]{hofstra2011}. \\~\\

\noindent
\textbf{The $\Pi_{\cF}$-completion.} The category $\Pi_{\cF}(\mE)$ has:
\begin{itemize}
    \item as \textbf{objects} pairs $(g:I \twoheadrightarrow J, \alpha)$ where $\alpha$ is an object of the fibre $\mE_I$;
    \item as \textbf{morphisms} triples $(f_0,f_1,\phi) : (g:X \twoheadrightarrow Y, \beta)\to (h:I \twoheadrightarrow J, \alpha)$ where 
\[\begin{tikzcd}
	\alpha & {f_1^*\alpha} && {f_0'^*\beta} & \beta \\
	X && {I \times_J Y} && Y \\
	&& I && J
	\arrow[Rightarrow, dashed, no head, from=1-1, to=2-1]
	\arrow[from=1-2, to=1-1, cart]
	\arrow["\phi", from=1-2, to=1-4, squiggly]
	\arrow[Rightarrow, dashed, no head, from=1-2, to=2-3]
	\arrow[from=1-4, to=1-5, cart]
	\arrow[Rightarrow, dashed, no head, from=1-4, to=2-3]
	\arrow[Rightarrow, dashed, no head, from=1-5, to=2-5]
	\arrow["g"', two heads, from=2-1, to=3-3]
	\arrow["{f_1}"', from=2-3, to=2-1]
	\arrow["{f_0'}", from=2-3, to=2-5]
	\arrow[two heads, from=2-3, to=3-3]
	\arrow["\lrcorner"{anchor=center, pos=0.125}, draw=none, from=2-3, to=3-5]
	\arrow["h", two heads, from=2-5, to=3-5]
	\arrow["{f_0}"', from=3-3, to=3-5]
\end{tikzcd}\]
is a diagram in $\mB$ and $\phi:\alpha\rightsquigarrow \beta$ is a vertical arrow in $\mE$ over $I \times_J Y$.
\end{itemize}
In this case, the idea is that an object $( g:X\twoheadrightarrow I, \alpha)$ is thought of as a predicate of the form $\forall i,x.\alpha(i,x)$.

The functor $\fibration{\Pi_{\cF}(\mE)}{\Pi_{\cF}(\mp)}{\mB}$ is defined by the assignments
$\Pi_{\cF}(\mp)( g:X \twoheadrightarrow I, \alpha):= I$ and $\Pi_{\cF}(\mp)(f_0,f_1,\phi):=f_0$. It is straightforward to check that this functor defines a fibration.

Again, employing the assumption that the display map category $\pair{\mB}{\cF}$ has $\cF$-dependent products, one can check that the fibration $\fibration{\Pi_{\cF}(\mE)}{\Pi_{\cF}(\mp)}{\mB}$ has $\cF$-products.

The following result presents the ``dual theorem'' of \Cref{thm_coprod_monad}:
\begin{theorem}
    The assignment $(\fibration{\mE}{\mp}{\mB},\pair{\mB}{\cF})\mapsto (\fibration{\Pi_{\cF}(\mE)}{\Pi_{\cF}(\mp)}{\mB},\pair{\mB}{\cF})$ extends to a 2-monad on the 2-category $\mathfrak{DispFib}$. The 2-category of pseudo-algebras is 2-equivalent to the 2-full sub category of $\mathfrak{DispFib}$ whose objects are pairs $(\fibration{\mE}{\mp}{\mB},\pair{\mB}{\cF}$ where $\mp$ has $\cF$-products and whose 1-cells are product preserving 1-cells of $\mathfrak{DispFib}$.
\end{theorem}

Comparing the $\Sigma_{\cF}$- with the $\Pi_{\cF}$-completion reveals a kind of symmetry: as in the case of the simple coproduct and product completions (see \cite[Prop. 3.11]{hofstra2011}), the $\Pi_{\cF}$-completion can formally be obtained by combining the $\Sigma_{\cF}$-completion with the fibrewise opposite (see \Cref{def_op_fibration}). In detail, we have the following correspondence:
\begin{proposition}\label{prop_product_via_coproduct_and_op}
    There is an isomorphism of fibrations    $\Pi_{\cF}(\mp)\cong (\Sigma_{\cF}(\mp^{(\mathrm{op})}))^{(\mathrm{op})}$, and this is natural in $\mp$.
\end{proposition}

\begin{example}
Let $\mB$ be a category with finite limits. Then, 
combining \Cref{prop_product_via_coproduct_and_op} with \Cref{ex_cod_fib_as_coprod_comp}, we have that the opposite of the codomain fibration $\fibration{\mB^{\to}}{\mathsf{cod}^{(\mathrm{op})}}{\mB}$ on $\mB$ is an instance of the $\Pi_{\mathsf{Mor}(\mB)}$-completion, \ie~$\mathsf{cod}^{(\mathrm{op})}\cong \Pi_{\mathsf{Mor}(\mB)}(\id_{\mB})$. 
\end{example}
\begin{example}
Let $\mB$ be a category with finite limits. Then, combining \Cref{prop_product_via_coproduct_and_op} with \Cref{ex_mono_fib_as_coprod_comp}, we have that the opposite of the monos-fibration $\fibration{\mathsf{Sub}(\mB)}{\mathsf{sub}^{(\mathrm{op})}}{\mB}$ is an instance of the $\Pi_{\mathsf{Mon}(\mB)}$-completion, \ie, $\mathsf{sub}^{(\mathrm{op})}\cong \Pi_{\mathsf{Mon}(\mB)}(\id_{\mB})$.
\end{example}
Taking advantage of the original intuition of Hosftra, who proved in~\cite{hofstra2011} that the Dialectica construction can be decomposed in terms of simple coproducts and simple products completions, we combine the $\Sigma_{\cF}$-completion with the $\Pi_{\cF}$-completion generalising the ordinary presentation of the so-called \emph{Dialectica fibration} $\dial{\mp}$ to the dependent case:

\begin{definition}[$\cF$-Dialectica fibration]\label{def_dependent_dialectica}
    Let $\pair{\mB}{\cF}$ be a display map category  with $\cF$-dependent coproducts and  let $\fibration{\mE}{\mp}{\mB}$ be a fibration. We define the $\cF$-\textbf{Dialectica fibration} as the fibration $\depdial{\mp}{\cF}:=\Sigma_{\cF}\Pi_{\cF}(\mp)$.
\end{definition}
The ordinary notion of Dialectica fibration can be then obtained as a particular instance of \Cref{def_dependent_dialectica} by taking $\cF$ to be the class of (cartesian) product projections. 
\begin{remark}
 Notice that when we consider a well-rooted display map category $\pair{\mB}{\cF}$ (see \Cref{def_well_rooted_disp_map_cat}) with $\cF$-dependent coproducts, then the fibre over $1$ of the $\cF$-Dialectica  fibration associated with the monos-fibration over $\mB$ provides the original notion of \emph{dependent Dialectica category} as introduced in \cite{dePaiva1989dialectica,depaiva1991dialectica}.

\end{remark}

\begin{remark}

    By \Cref{ex_mono_fib_as_coprod_comp}, we know that if $\mB$ is a category with finite limits, then the monos-fibration $\fibration{\mathsf{Sub}(\mB)}{\mathsf{sub}}{\mB}$ is an instance of the $\Sigma_{\mathsf{Mon}(\mB)}$-completion, namely $\mathsf{sub}\cong \Sigma_{\mathsf{Mon}(\mB)}(\id_{\mB})$. Therefore, we have that any $\cF$-Dialectica fibration $\depdial{\mathsf{sub}}{\cF}$ associated with the monos-fibration (assuming $\pair{\mB}{\cF}$ to be a display map category  with $\cF$-dependent coproducts), can easily presented as combinations of corpoducts and products completions of the identity fibration, \ie, $\depdial{\mathsf{sub}}{\cF}\cong \Sigma_{\cF}\Pi_{\cF}\Sigma_{\mathsf{Mon}(\mB)}(\id_{\mB})$. Similarly, by \Cref{ex_cod_fib_as_coprod_comp}, we have that any $\cF$-Dialectica fibration $\depdial{\mathsf{cod}}{\cF}$  associated with the codomain fibration $\fibration{\mB^{\to}}{\mathsf{cod}}{\mB}$  can be presented as  $\depdial{\mathsf{cod}}{\cF}\cong \Sigma_{\cF}\Pi_{\cF}\Sigma_{\mathsf{Mor}(\mB)}(\id_{\mB})$.
\end{remark}

\begin{example}[Polynomial functors]
Let $\pair{\mB}{\cF}$ be a display map category where $\mB$ is locally cartesian closed. One recovers the \emph{category $\mathbf{Poly}_{\cF}(\mB)$ of $\cF$-polynomial functors $\mB \to \mB$} as the fibre over $1 \in \mB$ of the fibration $\mathfrak{Dial}_\cF(\fibration{\cF}{\mathsf{cod}}{\mB})$. This plays a crucial role in the construction of Dialectica models of type theory in the work of von Glehn and Moss~\cite{vonGlehnPhD,moss2018,mossvonglehn2018}. The objects of $\mathbf{Poly}_{\cF}(\mB)$ are display maps $B \fibarr A \in \cF$, written type-theoretically as $\sum_{a:A}B(a) \fibarr A$. Such a map corresponds to the polynomial functor $\mB \to \mB, X \mapsto \sum_{a:A} X^{B(a)}$~\cite[Section~4.1]{vonGlehnPhD}. More generally, one can understand $\mathfrak{Dial}_\cF(\mp) \cong \Sigma_\cF \Pi_\cF(\mp)$ as the \emph{fibration of $\cF$-polynomials internal to $\mp$}, for a general fibration $\mp$~\cite{vonGlehnPhD,moss2022:polyTalk}. 
\end{example}

\section{Dependent Skolem fibrations}

In this section, we introduce the notion of dependent Skolem fibration, validating a principle analogous to \emph{Skolemisation}
\[ \forall x \exists y \phi(x,y) \cong \exists f \forall x \phi(x,fx).\]
The axioms for a Skolem fibration (to occur again later when introducing dependent Gödel fibrations) rely on the important notion of \emph{quantifier-free objects} that we study first.

This generalises the previous developments in~\cite{trotta_et_al:LIPIcs.MFCS.2021.87} from the simple to the dependent case. Besides, as a special class of Skolem fibrations we newly introduce \emph{Hilbert fibrations}, admitting operations analogous to Hilbert's \emph{$\epsilon$- and $\tau$-operators} from proof theory~\cite{hilbert1922,hilbert1923,ackermann1925}.

\subsection{Quantifier-free objects}
A first categorical description of the logical notion of \emph{existential-free objects} has been presented in the proof-irrelevant setting of Lawvere doctrines in the recent work \cite{maiettitrotta2023} by M.E.~Maietti and D.~Trotta and in the works \cite{Frey2020,Frey2014AFS} by J.~Frey.
In such a setting, and for a given class of display maps, the authors identify a universal property that an element of a doctrine has to satisfy in order to be considered ``free from existential quantifiers along display maps.''

Here we provide a proof-relevant generalisation of these notions. Since this further step of generality could make the reader loses the intuition behind the categorical definitions, we start by presenting a simple example in a non-fibrational setting that properly represents the picture we want to abstract.

Let us consider a locally small category $\mB$ with (set-indexed) sums. Recall, for example from \cite[Lem. 42]{CARBONI199879}, that an object $X$ of $\mB$ is said to be \emph{indecomposable}\footnote{The notion is originally due to Bunge~\cite{bungePhD} who called them \emph{abstractly exclusively unary objects}.} if its covariant hom-functor preserves sums, \ie, if the functor $\mB(X,-): \mB\to \set$ satisfies $\mB(X,\coprod_{i\in I}Y_i)\cong \coprod_{i\in I}\mB(X,Y_i)$. Notice that this property of $X$ can be presented in the following equivalent way: $X$ is indecomposable if and only if for every arrow $\arrow{X}{h}{\coprod_{i\in I}Y_i}$ there exist a unique element $\overline{i}\in  I$ (\ie, a function $\arrow{1}{\overline{i}}{I}$) and a unique arrow $\arrow{X}{\overline{h}}{Y_{\overline{i}}}$ such that the following diagram commutes
\[\begin{tikzcd}
	X && {\coprod_{i\in I}Y_i} \\
	& {Y_{\overline{i}}}
	\arrow["h", from=1-1, to=1-3]
	\arrow["{\overline{h}}"', from=1-1, to=2-2]
	\arrow["{\iota_{\overline{i}}}"', from=2-2, to=1-3]
\end{tikzcd}\]
where $\arrow{Y_{\overline{i}}}{\iota_{\overline{i}}}{\coprod_{i \in I}Y_i}$ is the canonical ``injection'' of the coproduct. 

In the following definition we generalise this notion of ``indecomposable object'' in the fibrational setting.

We fix a display map category $\pair{\mB}{\cF}$ closed under $\cF$-coproducts.

\begin{definition}[dependent $(\cF,\coprod)$-quantifier splitting objects]\label{def:gen-coprod-quant-splitting}
Let $\fibration{\mE}{\mP}{\mB}$ be a fibration with all $\cF$-coproducts. For $A\in \mB$, an object $\alpha \in \mE_A$ in the fibre is called \textbf{(dependent) $(\cF,\coprod)$-quantifier splitting} in case the following universal property holds:
given an object $\beta \in \mE_B$, in the fibre of some $B \in \mB$, together with a vertical map
\[h: \alpha \vertarr \coprod_u \beta\]
in $\mE_A$, where $u:B \tofib A$ is an arrow of $\cF$,
there uniquely exists the following:
\begin{itemize}
    \item a section
\[\begin{tikzcd}
	B && A
	\arrow["u"{description}, two heads, from=1-1, to=1-3]
	\arrow["g"{description}, curve={height=-12pt}, dotted, from=1-3, to=1-1]
\end{tikzcd}\]
of $u$ (\ie, a right inverse, not necessarily in $\cF$ itself);
    \item together with a vertical arrow $\overline{h}:\alpha \vertarr g^*\beta$ in $\mE_A$ such that the vertical arrow $h$ decomposes as
\begin{align}\label{eq:gen-coprod-quant-free}
    h = g^*\eta_\beta \circ \overline{h} 
\end{align}  
where $\eta : \id_{\mE_B} \Rightarrow u^* \coprod_u$ is the unit of the adjunction $\coprod_u \dashv u^*: \mE_A \to \mE_B$.
\end{itemize}
 Condition~\eqref{eq:gen-coprod-quant-free} means that the following diagram of vertical arrows in $\mE_A$ commutes (up to composition with the chosen isomorphism $g^*u^* \cong \id_{\mE_A}$, as indicated):
\[\begin{tikzcd}
	{\alpha} &&&& {g^*(u^*\coprod_u\beta) \mathrlap{\cong \coprod_u\beta}} \\
	&& {g^*\beta}
	\arrow["{g^*\eta_\beta}"{description}, squiggly, from=2-3, to=1-5]
	\arrow["h"{description}, squiggly, from=1-1, to=1-5]
	\arrow["{\overline{h}}"{description}, dashed, squiggly, from=1-1, to=2-3]
\end{tikzcd}\]
\end{definition}

\begin{remark}
    Notice that if $\alpha \in \mE_A$ is $(\cF,\coprod)$-quantifier splitting then we have that for every arrow $u:B \tofib A$ of $\cF$ having a section $g$ that the post-composition 
    \[g^*\eta_{\beta}\circ (-):\mE_A(\alpha,g^*\beta) \to \mE_A(\alpha,\coprod_u\beta)\]
    is monic for every $\beta \in \mE_B$.
\end{remark}
\begin{example}\label{ex:existential splitting family fibration}
Let us consider the family fibration $\fibration{\Fam (\mB)}{\mp}{\set}$, and let $\cF$ be the class of all the morphisms of $\set$. Recall from  \cite[Lem. 1.9.5]{Jacobs1999} that the family fibrations has coproducts (along all the arrows of $\set$) if and only if $\mB$ has set-indexed coproducts. In this setting the $(\set,\coprod)$-quantifier splitting elements are precisely those families $(X_i)_{i\in I}$ where every object $X_i$ is indecomposable in $\mB$.
We show this for the case $I$ is the terminal set $1=\{\bullet\}$, but the following argument can be easily generalised to an arbitrary set.

Let us consider an object $X$ of $\mB$, \ie, an object $(X_{\bullet})_{\bullet\in 1}$ of the fibre $\Fam (\mB)_{1}$, and let us consider another object of $\Fam (\mB)_{1}$ given by $\coprod_{i\in I}Y_i$ in $\mB$ where $\arrow{I}{!_I}{1}$ is the ``terminal function'' and $\coprod_{!_I}: {\Fam (\mB)_I}\to {\Fam (\mB)_1}$ is the right adjoint to the reindexing ${!_I^*}:{\Fam (\mB)_1}\to{\Fam (\mB)_I}$. By \Cref{def:gen-coprod-quant-splitting}, we have that $(X_{\bullet})_{\bullet \in I}$ is $(\cF,\coprod)$-quantifier splitting if and only if for every vertical arrow $(X_{\bullet})_{\bullet \in 1}\xrightarrow{h} \coprod_{!_I}(Y_i)_{i\in I}$ there exists a unique section $g: 1\to I$ of $!_I$ and a unique vertical arrow $\overline{h}: (X_{\bullet})_{\bullet\in 1} \to g^*(Y_i)_{i\in I}$ such that the diagram
\[\begin{tikzcd}
	(X_{\bullet})_{\bullet\in 1} && \coprod_{!_I}(Y_i)_{i\in I} \\
	& g^*(Y_{i})_{i\in I}
	\arrow["h", from=1-1, to=1-3]
	\arrow["{\overline{h}}"', from=1-1, to=2-2]
	\arrow["{\eta_{(Y_i)_{i\in I}}}"', from=2-2, to=1-3]
\end{tikzcd}\]
commutes, where $\eta:\id\to !_I^*\coprod_{!_I}$ is the unit of the adjunction $\coprod_{!_I}\dashv \; !_I^*$. Since the coproducts of the family fibration are given precisely by the coproducts of the category $\mB$ we can conclude that $(X_{\bullet})_{\bullet\in 1}$ is $(\cF,\coprod)$-quantifier splitting if and only if $X$ is indecomposable in $\mB$.

\end{example}
Notice that in general the property of being $(\cF,\coprod)$-quantifier-splitting is not stable under reindexings, \ie, if $\alpha\in \mE_A$ is $(\cF,\coprod)$-quantifier-splitting, the object $f^*\alpha$ is not  $(\cF,\coprod)$-quantifier-splitting in general. 

However, from a purely logical perspective where $(\cF,\coprod)$-quantifier-splittings aim to represent existential-free formulas, it is quite natural  to require this further condition of being ``stable under substitution.''

Hence, quantifier-splittings that are stable under reindexing are called \emph{quantifier-free elements}, according to the following definition:
\begin{definition}[dependent $(\cF,\coprod)$-quantifier-free objects]\label{def:gen-coprod-quant-free}
    Let $\fibration{\mE}{\mP}{\mB}$ be a fibration with all $\cF$-coproducts. For $A \in \mB$, an object $\alpha \in \mE_A$ in the fibre is called \textbf{(dependent) $(\cF,\coprod)$-quantifier-free} if for every arrow $f:{I}\to{A}$ in $\mB$, the reindexing $f^*\alpha$ is an $(\cF,\coprod)$-quantifier splitting.
\end{definition}

\begin{example}
Let us consider the family fibration $\fibration{\Fam (\mB)}{\mp}{\set}$, where $\mB$ is a category with set-indexed coproducts, and let $\cF$ be the class of all the morphisms of $\set$. In \Cref{ex:existential splitting family fibration} we show that an element $(X_i)_{i \in I}$ of the fibre $\Fam (\mB)_I$ is $(\cF,\coprod)$-quantifier-splitting if and only if every object $X_i$ is indecomposable in $\mB$. Since the action of the reindexing of the family fibrations does not change the objects of a family $(X_i)_{i \in I}$ but just the set-indexes, we have that if $f:J\to I$ is a function and $(X_i)_{i \in I}$ is an object of $\Fam (\mB)_I$ such that every $X_i$ is indecomposable, then every object of $f^*(X_i)_{i \in I}$ is indecomposable. Therefore, we have that every  $\Fam (\mB)_I$ is $(\cF,\coprod)$-quantifier-splitting is  $\Fam (\mB)_I$ is $(\cF,\coprod)$-quantifier-free.
\end{example}
\begin{definition}[enough $(\cF,\coprod)$-quantifier-free objects]
Let $\pair{\mB}{\cF}$ be a display map category. A fibration $\fibration{\mE}{\mP}{\mB}$ is said to have \textbf{enough $(\cF,\coprod)$-quantifier-free objects} if it has all $\cF$-coproducts and the following property holds: for all $I \in \mB$ and $\alpha\in \mE_I$ there exists some object $A \in \mB$, an arrow $f:A \twoheadrightarrow I$ in $\cF$, and $\beta \in \mE_A$ $(\cF,\coprod)$-quantifier-free object such that $\alpha \cong \coprod_f(\beta)$ .
\end{definition}
We conclude this section by presenting the dual of the previous notions for the case of dependent products:

\begin{definition}[dependent $(\cF,\prod)$-quantifier splitting objects]\label{def:gen-prod-quant-splitting}
Let $\fibration{\mE}{\mP}{\mB}$ be a fibration with all $\cF$-products. For $A\in \mB$, an object $\alpha \in \mE_A$ in the fibre is called \textbf{(dependent) $(\cF,\prod)$-quantifier splitting} in case the following universal property holds:
given an object $\beta \in \mE_B$, in the fibre of some $B \in \mB$, together with a vertical map
\[h: \prod_u \beta \vertarr \alpha\]
in $\mE_A$, where $u:B \tofib A$ is an arrow of $\cF$,
there uniquely exists the following:
\begin{itemize}
    \item a section
\[\begin{tikzcd}
	B && A
	\arrow["u"{description}, two heads, from=1-1, to=1-3]
	\arrow["g"{description}, curve={height=-12pt}, dotted, from=1-3, to=1-1]
\end{tikzcd}\]
of $u$ (\ie, a right inverse, not necessarily in $\cF$ itself);
    \item together with a vertical arrow $\overline{h}:g^*\beta\vertarr \alpha$ in $\mE_A$ such that the vertical arrow $h$ decomposes as
\begin{align}\label{eq:gen-prod-quant-free}
    h =   \overline{h} \circ g^*\varepsilon_\beta
\end{align}  
where $\varepsilon :  u^* \prod_u\Rightarrow \id_{\mE_B}  $ is the counit of the adjunction $u^*\dashv \prod_u : \mE_A \to \mE_B$.
\end{itemize}
 Condition~\eqref{eq:gen-prod-quant-free} means that the following diagram of vertical arrows in $\mE_A$ commutes:
\[\begin{tikzcd}
	g^*(u^*\prod_u\beta) \cong \prod_u\beta &&&& {\alpha}\\
	&& {g^*\beta}
	\arrow["{\overline{h}}"{description}, squiggly, from=2-3, to=1-5]
	\arrow["h"{description}, squiggly, from=1-1, to=1-5]
	\arrow["{g^*\varepsilon_\beta}"{description}, dashed, squiggly, from=1-1, to=2-3]
\end{tikzcd}\]
\end{definition}

\begin{definition}[dependent $(\cF,\prod)$-quantifier-free objects]\label{def:gen-prod-quant-free}
    Let $\fibration{\mE}{\mP}{\mB}$ be a fibration with all $\cF$-products. For $A \in \mB$, an object $\alpha \in \mE_A$ in the fibre is called \textbf{(dependent) $(\cF,\prod)$-quantifier-free} if for every arrow $f:{I}\to{A}$ in $\mB$, the reindexing $f^*\alpha$ is $(\cF,\prod)$-quantifier splitting.
\end{definition}

\begin{definition}[enough $(\cF,\prod)$-quantifier-free objects]
Let $\pair{\mB}{\cF}$ be a display map category. A fibration $\fibration{\mE}{\mP}{\mB}$ is said to have \textbf{enough $(\cF,\prod)$-quantifier-free objects} if it has all $\cF$-products and the following property holds: for all $I \in \mB$ and $\alpha\in \mE_I$ there exists some object $A \in \mB$, an arrow $f:A \twoheadrightarrow I$ in $\cF$, and $\beta \in \mE_A$ an $(\cF,\prod)$-quantifier-free object such that $\alpha \cong \prod_f(\beta)$.
\end{definition}

\subsection{Fibrations equipped with Hilbert $\epsilon$- and $\tau$-operators}

The notions of $(\cF,\coprod)$-quantifier-free and  $(\cF,\prod)$-quantifier-free elements introduced in the previous section allow us to formally introduce in the language of fibrations a generalisation of two concepts known as Hilbert's $\epsilon$- and $\tau$- operators, see \cite{bell} and \cite{Devidi95} for more details. A first categorical presentation of the $\epsilon$-operators was introduced in the proof-irrelevant setting of Lawvere's doctrines by M.E. Maietti, F. Pasquali and G. Rosolini in  \cite{maiettipasqualirosolini}. 

We briefly recall that, from a purely logical perspective, Hilbert's $\epsilon$-calculus is an extension of first-order logic with $\epsilon$-\emph{operators} representing witness functions of existential quantifiers: the intuition is that, given a first-order language, for every formula $A$ and variable $x$, we add a term $\epsilon x A$ representing \emph{some} $x$ satisfying $A$. These $\epsilon$-terms are governed by the so-called \emph{transfinite axiom}:
\[ A(x)\to A(\epsilon x A).\]
Such an extension provides a quantifier-free calculus because, \emph{classically}, we have that
\begin{align*}
    &\exists x A(X)\equiv A(\epsilon x A)\\
    &\forall x A(X)\equiv A(\epsilon x (\neg A)).
\end{align*}
Notice that, to properly apply such an approach in an intuitionistic setting, one needs to require also a dual-notion of $\epsilon$-operators since the previous second equivalence is not justified intuitionistically. This is precisely the intuition behind the notion Hilbert's $\tau$-operators: in this case, for every formula $A$ and variable $x$, we have to extend our first-order \emph{intuitionistic} language also with a term $\tau x A$ representing a dual notion of $\epsilon$-terms. Extending an intuitionistic language with both $\epsilon$- and $\tau$-operatos provides a quantifier-free calculus because
\begin{align*}
    &\exists x A(X)\equiv A(\epsilon x A)\\
    &\forall x A(X)\equiv A(\tau x  A).
\end{align*}

In the following definitions, we aim to present a fibrational account for these notions, generalising the notion introduced in \cite[Def. 5.10]{maiettipasqualirosolini} of doctrine equipped with Hilbert's $\epsilon$-operators.

To achieve this goal, we employ the idea and the characterisation presented in the language of Lawvere doctrines in \cite{maiettitrotta2023}: in particular, in \cite[Thm. 5.1]{maiettitrotta2023} the authors show that a given existential doctrine is equipped with Hilbert's $\epsilon$-operators in the sense of \cite{maiettipasqualirosolini} if and only if every element of the doctrine is $\exists$-free (\ie~the doctrine is isomorphic to the existential completion of itself). 

Motivated by this result in the proof-irrelevant setting, we introduce the following notion in the language of fibrations:

\begin{definition}[dependent Hilbert $\epsilon$-fibration]
 Let $\pair{\mB}{\cF}$ be a display map category. A fibration $\fibration{\mE}{\mP}{\mB}$ is called a \textbf{(dependent) Hilbert $\epsilon$-fibration} if
\begin{itemize}
    \item $\pair{\mB}{\cF}$ is a display map category with $\cF$-dependent coproducts;
    \item the fibration $\mP$ has fibred coproducts along $\cF$;
    \item every object of the fibration $\mP$ is a $(\cF, \coprod)$-quantifier-free object.
\end{itemize}
\end{definition}

\begin{theorem}\label{thm_char_Hilbert_epsilon}
    Let $\pair{\mB}{\cF}$ be a display map category  with $\cF$-dependent coproducts, and $\fibration{\mE}{\mP}{\mB}$ a fibration with fibred coproducts along $\cF$. Then $\mP$ is a dependent Hilbert $\epsilon$-fibration if and only if every element $\alpha\in \mE_I$ and every display map $f:I \twoheadrightarrow J$ there uniquely exist an arrow $\epsilon_{\alpha,f}:J \to I$ and a vertical map 
    \[\overline{\epsilon}_{\alpha,f}:   \coprod_f \alpha\vertarr (\epsilon_{\alpha,f})^*\alpha \]
    such that
    \begin{enumerate}
        \item $f\circ \epsilon_{\alpha,f}=\id_I$;
        \item   $ \id = (\epsilon_{\alpha,f})^*\eta_\alpha \circ \overline{\epsilon}_{\alpha,f}$,  where $\eta : \id_{\mE_I} \Rightarrow f^* \coprod_f$ is the unit of the adjunction $\coprod_f \dashv f^*$;
        \item if a vertical arrow $u:\alpha \vertarr \coprod_g \beta $ admits a decomposition 
\[\begin{tikzcd}
	\alpha && {\coprod_g\beta} \\
	& {t^*(\beta)}
	\arrow["u", squiggly, from=1-1, to=1-3]
	\arrow["h"', squiggly, from=1-1, to=2-2]
	\arrow["{t^*(\eta_{\beta})}"', squiggly, from=2-2, to=1-3]
\end{tikzcd}\]
for some vertical arrow $h$ and some section $t$ of $u$, then  $t=\epsilon_{\beta,g}$ and $h=\bar{\epsilon}_{\beta,g}\circ u$.
        
    \end{enumerate}
\end{theorem}
\begin{proof}
   $(\Rightarrow)$ By definition of dependent Hilbert $\epsilon$-fibrations, we have that every element is a $(\cF, \coprod)$-quantifier-free object. Therefore, we obtain ponits $(1)$ and $(2)$ by applying the definition of $(\cF, \coprod)$-quantifier-free object to the identity vertical arrow on $\id:  \coprod_f \alpha\vertarr  \coprod_f \alpha$. To show the last point, let us consider a factorization 
   \[\begin{tikzcd}
	\alpha && {\coprod_g\beta} \\
	& {t^*(\beta)}
	\arrow["u", squiggly, from=1-1, to=1-3]
	\arrow["h"', squiggly, from=1-1, to=2-2]
	\arrow["{t^*(\eta_{\beta})}"', squiggly, from=2-2, to=1-3]
\end{tikzcd}\]
of  $u$, with $t$ section of $u$. By hypothesis, every element is $(\cF, \coprod)$-quantifier-free, so this factorization, that is of the correct form as required in the definition of  $(\cF, \coprod)$-quantifier-free elements, is unique. But now notice that also
\[\begin{tikzcd}
	\alpha && {\coprod_g\beta} && {\coprod_g\beta} \\
	& {t^*(\beta)} && {\epsilon_{\beta,g}^*(\beta)}
	\arrow["u", squiggly, from=1-1, to=1-3]
	\arrow["h"', squiggly, from=1-1, to=2-2]
	\arrow["{t^*(\eta_{\beta})}"', squiggly, from=2-2, to=1-3]
	\arrow["\bar{\epsilon}_{\beta,g}"',squiggly, from=1-3, to=2-4]
	\arrow["\epsilon_{\beta,g}^*(\eta_\beta)"',squiggly, from=2-4, to=1-5]
	\arrow["\id",squiggly, from=1-3, to=1-5]
\end{tikzcd}\]
   provides another factorization of $u$, and it has the correct shape as required in the definition of $(\cF, \coprod)$-quantifier-free elements. Therefore, we can deduce that $\epsilon_{\beta,g}=t$ and that 
   $h=\bar{\epsilon}_{\beta,g} \circ {\epsilon}_{\beta,g}^*(\eta_\beta)\circ h= \bar{\epsilon}_{\beta,g}\circ u$. This concludes the proof of point $(3)$.

   \noindent
   
   $(\Leftarrow)$ Let $\alpha$ be an arbitrary object of $\mE_I$, and let us consider a vertical arrow $u: \alpha \vertarr \coprod_g \beta$ where  $g:A\twoheadrightarrow I$ is a display map. To show that $\alpha$ is  a $(\cF, \coprod)$-quantifier splitting it is enough to apply our assumptions to the object $ \coprod_g \beta$.  In fact, by hypothesis, we obtain that there uniquely exist an arrow  $\epsilon_{\beta,g}:I\to A$ and a vertical arrow $\overline{\epsilon}_{\beta,g}:   \coprod_g \beta\vertarr (\epsilon_{\beta,g})^*\beta $ satisfying the conditions $(1)$ and $(2)$, and we can use these arrows to define a vertical arrow $\overline{\epsilon}_{\beta,g}\circ u:   \alpha \vertarr (\epsilon_{\beta,g})^*\beta $. By $(1)$, we have that $g \circ \epsilon_{\beta,g}=\id$, and by $(2)$, that $ u = (\epsilon_{\beta,g})^*\eta_\beta \circ ( \overline{\epsilon}_{\beta,g}\circ u )$, \ie, such that the diagram
\[\begin{tikzcd}
	\alpha && {\coprod_g \beta} \\
	{\coprod_g \beta} && {(\epsilon_{\beta,g})^*\beta}
	\arrow["u", squiggly, from=1-1, to=1-3]
	\arrow["u"', squiggly, from=1-1, to=2-1]
	\arrow["{\overline{\epsilon}_{\beta,g}}"', squiggly, from=2-1, to=2-3]
	\arrow["{(\epsilon_{\beta,g})^*\eta_\beta}"', squiggly, from=2-3, to=1-3]
\end{tikzcd}\]
commutes. Finally, we have that such a factorization is unique by point $(3)$.
   So we have proved that every object $\alpha$ satisfies the conditions required in \Cref{def:gen-coprod-quant-splitting}, \ie, every element of the fibration is a $(\cF, \coprod)$-quantifier splitting. Therefore, since every object is $(\cF, \coprod)$-quantifier splitting, we can conclude that every object is $(\cF, \coprod)$-quantifier-free, \ie, that $\mP$ is a dependent Hilbert $\epsilon$-fibration.
\end{proof}
The previous characterisation allows us to prove in the language of fibrations the desired feature of the $\epsilon$-calculus, namely that $\exists x A(X)\equiv A(\epsilon x A)$:
\begin{cor}
    Let $\pair{\mB}{\cF}$ be a display map category  with $\cF$-dependent coproducts, and $\fibration{\mE}{\mP}{\mB}$ a dependent Hilbert $\epsilon$-fibration. Then for every object $\alpha\in \mE_I$ and every display map $f:I \twoheadrightarrow J$, we have that there exists a vertical arrow $\overline{\epsilon}_{\alpha,f}:   \coprod_f \alpha\vertarr (\epsilon_{\alpha,f})^*\alpha $ and it is an isomorphism.
\end{cor}
\begin{proof}
    By \Cref{thm_char_Hilbert_epsilon} we have that $ \id = (\epsilon_{\alpha,f})^*\eta_\alpha \circ \overline{\epsilon}_{\alpha,f}$, so we need to prove $ \id =  \overline{\epsilon}_{\alpha,f} \circ (\epsilon_{\alpha,f})^*\eta_\alpha $. To show this, it is enough to observe that the diagram
\[\begin{tikzcd}
	{(\epsilon_{\alpha,f})^*\alpha} && {\coprod_f \alpha} \\
	& {(\epsilon_{\alpha,f})^*(\alpha)}
	\arrow["{(\epsilon_{\alpha,f})^*(\eta_{\alpha})}", squiggly, from=1-1, to=1-3]
	\arrow["{\bar{\epsilon}_{\alpha,f}\circ (\epsilon_{\alpha,f})^*(\eta_{\alpha})}"', squiggly, from=1-1, to=2-2]
	\arrow["{(\epsilon_{\alpha,f})^*(\eta_{\alpha})}"', squiggly, from=2-2, to=1-3]
\end{tikzcd}\]
commutes, since $ \id = (\epsilon_{\alpha,f})^*\eta_\alpha \circ \overline{\epsilon}_{\alpha,f}$. In fact, from this we can conclude that $ \id =  \overline{\epsilon}_{\alpha,f} \circ (\epsilon_{\alpha,f})^*\eta_\alpha $ from the fact that  ${(\epsilon_{\alpha,f})^*\alpha}$ is an $(\cF, \coprod)$-quantifier-free element and hence the factorization ${(\epsilon_{\alpha,f})^*(\eta_{\alpha})}$ is unique.
\end{proof}
We conclude this section by presenting the notion of dependent Hilbert $\tau$-fibration and its characterisation. Notice that since this notion is precisely the dual of the previous one, all the results and proofs presented for the case of $\epsilon$-operators can be dualized to the case of $\tau$-operators.
\begin{definition}[dependent Hilbert $\tau$-fibration]
 Let $\pair{\mB}{\cF}$ be a display map category. A fibration $\fibration{\mE}{\mP}{\mB}$ is called a \textbf{(dependent) Hilbert $\tau$-fibration} if
\begin{itemize}
    \item $\pair{\mB}{\cF}$ is a display map category with $\cF$-dependent coproducts;
    \item the fibration $\mP$ has fibred products along $\cF$;
    \item every object of the fibration $\mP$ is a $(\cF, \prod)$-quantifier-free object.
\end{itemize}
\end{definition}
\begin{theorem}\label{thm_char_Hilbert_tau}
    Let $\pair{\mB}{\cF}$ be a display map category  with $\cF$-dependent coproducts, and $\fibration{\mE}{\mP}{\mB}$ a fibration with fibred products along $\cF$. Then $\mP$ is a dependent Hilbert $\tau$-fibration if and only if for every element $\alpha\in \mE_I$ and every display map $f:I \twoheadrightarrow J$ there uniquely exist an arrow $\tau_{\alpha,f}:J \to I$ and a vertical map 
    \[\overline{\tau}_{\alpha,f}: (\tau_{\alpha,f})^*\alpha \vertarr  \prod_f \alpha \]
    such that
    \begin{enumerate}
        \item $f\circ \tau_{\alpha,f}=\id_I$;
        \item   $ \id = \overline{\tau}_{\alpha,f} \circ (\tau_{\alpha,f})^*\varepsilon_\alpha $,  where $\varepsilon :  f^* \prod_f \Rightarrow \id_{\mE_I}$ is the counit of the adjunction $f^* \dashv \prod_f$;
        \item if a vertical arrow $u: \prod_g \beta \vertarr \alpha $ admits a decomposition 
\[\begin{tikzcd}
	{\prod_g\beta} && \alpha \\
	& {t^*(\beta)}
	\arrow["u", squiggly, from=1-1, to=1-3]
	\arrow["{t^*(\varepsilon_{\beta})}"', squiggly, from=1-1, to=2-2]
	\arrow["h"', squiggly, from=2-2, to=1-3]
\end{tikzcd}\]
for some vertical arrow $h$ and some section $t$ of $u$, then  $t=\tau_{\beta,g}$ and $h=u \circ \bar{\tau}_{\beta,g}$.
        
    \end{enumerate}
\end{theorem}
\begin{cor}
    Let $\pair{\mB}{\cF}$ be a display map category  with $\cF$-dependent coproducts, and $\fibration{\mE}{\mP}{\mB}$ a dependent Hilbert $\tau$-fibration. Then for every object $\alpha\in \mE_I$ and every display map $f:I \twoheadrightarrow J$, we have that there exists a vertical arrow $\overline{\tau}_{\alpha,f}:    (\tau_{\alpha,f})^*\alpha \vertarr \prod_f \alpha$ and it is an isomorphism.
\end{cor}

\subsection{Dependent Skolem fibrations}

Abstracting from the concrete class of product projections to an arbitrary class of display maps we are led to the following generalisation of the notion of Skolem fibration introduced in~\cite[Definition~2.7]{trotta_et_al:LIPIcs.MFCS.2021.87}.

\begin{definition}[Dependent Skolem fibration]
 Let $\pair{\mB}{\cF}$ be a display map category. A fibration $\fibration{\mE}{\mP}{\mB}$ is called a \textbf{(dependent) Skolem fibration} if
\begin{itemize}
    \item its base category $\mB$ has dependent products along $\cF$;
    \item the fibration $\mP$ has fibred products along $\cF$;
    \item the fibration $\mP$ has enough $(\cF, \coprod)$-quantifier-free objects;\footnote{In particular, $\mP$ has fibred coproducts along $\cF$.}
    \item $(\cF,\coprod)$-quantifier-free objects are stable under $\cF$-products, \ie, if for any $(\cF,\coprod)$-quantifier-free object $\alpha \in \mE_I$, $I \in \mB$, the object $\prod_f(\alpha) \in \mE_J$ is $(\cF, \coprod)$-quantifier-free, too, for any map $f:I \twoheadrightarrow J$ in $\cF$.
\end{itemize}
\end{definition}

As a convention, we will often abbreviate $\mB$ for $\pair{\mB}{\cF}$.

Again, as in the remark after~\cite[Definition~2.7]{trotta_et_al:LIPIcs.MFCS.2021.87}, we get that by the last condition, given a dependent Skolem fibration $\fibration{\mE}{\mP}{\mB}$ its full subfibration $\fibration{\mE'}{\mP'}{\mB}$ of $(\coprod,\cF)$-quantifier-free objects has fibred $\cF$-products.
\begin{example}
    Every (dependent) Hilbert $\epsilon$-fibration with fibred products and coproducts along $\cF$ (and whose base category has dependent products along $\cF$) is a Skolem fibration.
\end{example}
Dependent Skolem fibrations validate a Skolem principle, generalising the one from~\cite[Proposition~2.8]{trotta_et_al:LIPIcs.MFCS.2021.87}. Versions of this principle have been established in various related settings~\cite{awodey1995axiom,joyal2017notesclanstribes,hofstra2011,vonGlehn2018polynomials,ccnw2024}.

\begin{theorem}[Dependent Skolemisation]
Let $\pair{\mB}{\cF}$ be a display map category with units and $\fibration{\mE}{\mP}{\mB}$ a dependent Skolem fibration over it. Let $g:A \fibarr S$ and $f:B \fibarr A$ be maps in $\cF$. Consider the $\cF$-dependent product of $f$ along $g$, as given by the diagram:
\[\begin{tikzcd}[column sep=huge, row sep=huge]
	B & Z & X \\
	& A & S
	\arrow["f"', two heads, from=1-1, to=2-2]
	\arrow["g"', two heads, from=2-2, to=2-3]
	\arrow[two heads, from=1-2, to=2-2]
	\arrow["h", two heads, from=1-3, to=2-3]
	\arrow["{g'}", two heads, from=1-2, to=1-3]
	\arrow["e"', from=1-2, to=1-1]
	\arrow["\lrcorner"{anchor=center, pos=0.125}, draw=none, from=1-2, to=2-3]
\end{tikzcd}\]
Then, there is a vertical isomorphism
\[ \prod_g \coprod_f(\beta) \cong \coprod_h \prod_{g'} e^*(\beta)\]
for each $\beta \in \mE_B$.
\end{theorem}

We remark that this in general does \emph{not} seem to give rise to a natural isomorphism of functors.

\begin{proof}
We generalise the proof from~\cite[Proposition~2.8]{trotta_et_al:LIPIcs.MFCS.2021.87}. The difference is that we replace cartesian projections by the given  class $\cF$ of display maps. The strategy is as follows: we first produce a family of bijections
\[ \Phi_{\sigma,\beta} : \hom_{\mE_A}\Big(g^*\sigma,\coprod_f(\beta)\Big) \to \hom_{\mE_S}\Big(\sigma,\coprod_h \prod_{g'}e^*(\beta)\Big) \]
with $\beta \in \mE_B$ and $(\cF,\coprod)$-quantifier-free $\sigma \in \mE_S$.
We then construct an inverse $\Psi_{\sigma,\beta}$. Finally we show how to lift this to the case of general elements $\sigma \in \mE_S$.

\noindent
\textbf{1) Construction of $\Phi$.}
Let $\beta \in \mE_B$. Assume $\sigma \in \mE_S$ is a $(\cF,\coprod)$-quantifier-free object.

Let $m:g^*(\sigma) \vertarr \coprod_f(\beta)$ be a vertical arrow in $\mE_A$. By quantifier-freeness of $\sigma$ the element $g^*(\sigma)$ is a $(\cF,\coprod)$-quantifier splitting. Hence, there uniquely exist a section $u:A \to B$ in $\mB$ of $f$ and a vertical arrow $\overline{m} \in \mE_A(g^*(\sigma), u^*(\beta))$ such that $m= u^*\eta_\beta \circ \overline{m}$, for $\eta : \id_{\mE_B} \Rightarrow f^* \coprod_f$. Since $\mB$ has $\cF$-dependent products and $\cF$ has units, we get induced maps $k:S \to X$ and $k':A \to Z$ as follows:
\[\begin{tikzcd}[sep=large]
	&&& A &&& S \\
	B && Z & {} && X \\
	\\
	&& A &&& S
	\arrow["{h'}"{description}, two heads, from=2-3, to=4-3]
	\arrow["g"{description}, two heads, from=4-3, to=4-6]
	\arrow["h"{description}, two heads, from=2-6, to=4-6]
	\arrow["e"{description}, from=2-3, to=2-1]
	\arrow["f"{description}, two heads, from=2-1, to=4-3]
	\arrow["u"{description}, from=1-4, to=2-1]
	\arrow["{k'}"{description}, dashed, from=1-4, to=2-3]
	\arrow["{\exists! \,k}"{description}, dashed, from=1-7, to=2-6]
	\arrow["g"{description}, two heads, from=1-4, to=1-7]
	\arrow[curve={height=-12pt}, Rightarrow, no head, from=1-7, to=4-6]
	\arrow["\lrcorner"{anchor=center, pos=0.125}, draw=none, from=1-4, to=4-6]
	\arrow[curve={height=-12pt}, Rightarrow, no head, from=1-4, to=4-3]
	\arrow["\lrcorner"{anchor=center, pos=0.125}, draw=none, from=2-3, to=4-6]
	\arrow["{g'}"{description}, two heads, from=2-3, to=2-6, crossing over]
\end{tikzcd}\]
Note that $k = k_{e,f,g,h,\sigma,\beta,m}$ is uniquely determined with the property of making the respective subdiagrams commute, and so is $k'$.
For $\overline{m}\in  \mE_A(g^*(\sigma),u^*(\beta))$, consider the adjoint transpose $\overline{m}^\flat$ across the adjunction $g^* \dashv \prod_g$ given by 
\[\begin{tikzcd}
	\sigma && {\prod_g(k')^*e^*\beta \cong \prod_gu^*\beta} \\
	{\prod_g g^*\sigma}
	\arrow[squiggly, "{\overline{m}^\flat}", from=1-1, to=1-3]
	\arrow[squiggly, "{\kappa_\sigma}"', from=1-1, to=2-1]
	\arrow[squiggly, "{\prod_g \overline{m}}"', from=2-1, to=1-3]
\end{tikzcd}\]
in $\mE_S$. Since
\[\begin{tikzcd}
	A & S \\
	Z & X
	\arrow["{k'}"', from=1-1, to=2-1]
	\arrow["{g'}"', two heads, from=2-1, to=2-2]
	\arrow["g", two heads, from=1-1, to=1-2]
	\arrow["k", from=1-2, to=2-2]
	\arrow["\lrcorner"{anchor=center, pos=0.125}, draw=none, from=1-1, to=2-2]
\end{tikzcd}\]
is a pullback with $g,g' \in \cF$, we get from the Beck--Chevalley condition that
\[\prod_g(k')^* \cong k^* \prod_{g'}.\]
Consider the unit $\nu: \id_{\mE_A} \Rightarrow h^* \coprod_h$. Applying $k^*$ to  $\nu_\alpha: \alpha \rightsquigarrow h^* \coprod_h(\alpha)$, for any $\alpha \in \mE_A$, yields\footnote{In due course, we will often suppress the isomorphisms mediated by cartesian liftings and the BCC so as to not further complicate notation.}
\[k^*\nu_\alpha: k^*(\alpha) \vertarr k^*h^*\coprod_h(\alpha) \cong \coprod_h(\alpha).\]
Now, let $\alpha := \prod_{g'}e^*(\beta)$. Then, we find
\[ k^*\alpha = k^* \prod_{g'}e^*(\beta) \stackrel{\text{BCC}}{\cong} \prod_g (k')^* e^*(\beta) = \prod_g u^*(\beta). \]
This means that
\[ \overline{m}^\flat: \sigma \vertarr \prod_g u^*(\beta) \cong k^*(\alpha). \]
Postcomposing with $k^*\nu_\alpha$ yields an arrow
\[\begin{tikzcd}
	\sigma && {k^*(\alpha)} && {k^*h^*\coprod_h(\alpha) \cong \coprod_h(\alpha)=\coprod_h \prod_{g'}e^*(\beta)}
	\arrow[squiggly, "{\overline{m}^\flat}", from=1-1, to=1-3]
	\arrow[squiggly, "{k^*\nu_\alpha}", from=1-3, to=1-5]
\end{tikzcd}\]
and we define, up to isomorphism,
\[  \Phi_{\sigma,\beta} : \hom_{\mE_S}\Big(g^*\sigma,\coprod_f(\beta)\Big) \to \hom_{\mE_B}\Big(\sigma,\coprod_h \prod_{g'}e^*(\beta)\Big), \quad m \mapsto k^*\nu_{\prod_{g'}e^*(\beta)} \circ \overline{m}^\flat, \]
or more verbosely:
\[\begin{tikzcd}
	& {\mE_B(g^*\sigma,\coprod_f \beta)} && {\mE_B(\sigma,\coprod_h\prod_{g'}e^* \beta)} \\
	{\big( g^*\sigma} & {\coprod_f \beta\big)} && {\big(\sigma} && {k^*h^*\coprod_h\prod_{g'}e^*\beta \cong\coprod_h\prod_{g'}e^*\beta\big)} \\
	&&& {\prod_gu^*\beta \cong k^*\prod_{g'}e^*\beta}
	\arrow["{\Phi_{\sigma,\beta}}", from=1-2, to=1-4]
	\arrow[squiggly, "{\Phi_{\sigma,\beta}(m)}", from=2-4, to=2-6]
	\arrow["{\overline{m}^\flat}"', squiggly, from=2-4, to=3-4]
	\arrow["{k^*\nu_{\prod_{g'} e^*\beta}}"', squiggly, from=3-4, to=2-6]
	\arrow[maps to, from=2-2, to=2-4]
	\arrow[squiggly, "m", from=2-1, to=2-2]
\end{tikzcd}\]

\noindent
\textbf{2) Construction of $\Psi$.} We now have to construct an inverse for $\Phi_{\sigma,\beta}$. We claim that this is given by the following family of maps, where the composition is supposed to be read up to some chosen canonical isomorphisms:
\[ \Psi_{\sigma,\beta} : \hom_{\mE_B}\Big(\sigma,\coprod_h \prod_{g'}e^*(\beta)\Big) \to \hom_{\mE_S}\Big(g^*\sigma,\coprod_f(\beta)\Big), r \mapsto u^*\eta_\beta \circ \widetilde{r}^\sharp. \]
We are to describe its action on arrows. Starting with $r:\sigma \to \coprod_h(\prod_{g'} e^*\beta)$, letting $\gamma := \prod_{g'} e^*\beta$, we obtain by the quantifier splitting a map $\widetilde{r}$, such that:
\[\begin{tikzcd}
	\sigma && {k^*h^*\coprod_h(\prod_{g'}e^*\beta)} \\
	& {k^*\prod_{g'}e^*\beta}
	\arrow[squiggly, "r", from=1-1, to=1-3]
	\arrow[squiggly, "{\widetilde{r}}"', from=1-1, to=2-2]
	\arrow[squiggly, "{k^*\nu_\gamma}"', from=2-2, to=1-3]
\end{tikzcd}\]
By the BCC, we have
\[ \prod_g (k')^* \cong k^* \prod_{g'}\]
and so
\[ \widetilde{r}: \sigma \vertarr k^*\prod_{g'} e^*\beta \cong \prod_{g}(k')^*e^*\beta. \]
Its left adjoint transpose is given by:

\[\begin{tikzcd}
	{g^*\sigma} && {(k')^*e^*\beta} \\
	{g^*\coprod_g(k')^*e^*\beta}
	\arrow[squiggly, "{\widetilde{r}^\sharp}", from=1-1, to=1-3]
	\arrow[squiggly, "{g^*\widetilde{r}}"', from=1-1, to=2-1]
	\arrow[squiggly, "{\varepsilon_{(k')^*e^*\beta}}"', from=2-1, to=1-3]
\end{tikzcd}\]

In sum, the candidate inverse map is then given by:
\[\begin{tikzcd}
	& {\mE_B(\sigma,\coprod_h \prod_{g'}e^*\beta)} && {\mE_S(g^*\sigma,\coprod_f \beta)} \\
	{\big( \sigma} & {\coprod_h \prod_{g'}e^*\beta\big)} && {\big(g^*\sigma} && {u^*f^*\coprod_f\beta\cong \coprod_f\beta\big)} \\
	&&& {(k')^*e^*\beta\cong u^*\beta}
	\arrow["{\Psi_{\sigma,\beta}}", from=1-2, to=1-4]
	\arrow[squiggly, "{\Psi_{\sigma,\beta}(r)}", from=2-4, to=2-6]
	\arrow[squiggly, "{\widetilde{r}^\sharp}"', from=2-4, to=3-4]
	\arrow[squiggly, "{u^*\eta_\beta}"', from=3-4, to=2-6]
	\arrow[maps to, from=2-2, to=2-4]
	\arrow[squiggly, "r", from=2-1, to=2-2]
\end{tikzcd}\]
\noindent
\textbf{3) Mutual inverseness.} We claim that $\Psi_{\sigma,\beta}$ and $\Phi_{\sigma,\beta}$ are inverse to each other. Let us suppress the indices in the following. We explicitly treat the case for $\Psi \circ \Phi = \id$, the case of $\Phi \circ \Psi = \id$ is analogous.

First, we claim that $\widetilde{\Phi(m)}  = \overline{m}^\flat$, \ie, we want to verify that:
\[\begin{tikzcd}
	\sigma && {k^*h^*\coprod_h(\prod_{g'}e^*\beta)} \\
	{k^*\prod_{g'}e^*\beta \cong \prod_{g}(k')^*e^*\beta}
	\arrow[squiggly, "{\Phi(m)}", from=1-1, to=1-3]
	\arrow[squiggly, "{k^*\nu_{\prod_{g'}e^*\beta}}"', from=2-1, to=1-3]
	\arrow[squiggly, "{\widetilde{\Phi(m)}}"', from=1-1, to=2-1]
\end{tikzcd}\]
But, by definition of $\Phi$, we have $\Phi(m) = k^*\nu_{\prod_{g'}e^*\beta} \circ \overline{m}^\flat$, and by uniqueness of the section-factorization pair due to quantifier-freeness this yields $\widetilde{\Phi(m)} = \overline{m}^\flat$.

Now, we recall the general formula for $\Psi$: let $r \colon \sigma \vertarr \coprod_h \prod_{g'} e^* \beta$, then
\[\begin{tikzcd}
	{g^{*}\sigma} && {g^{*}\prod_g(k')^{*}e^{*}\beta} && {(k')^{*} e^{*} \beta \cong u^{*}\beta} && {u^{*}f^{*}\coprod_f\beta \cong\coprod_f\beta.}
	\arrow[squiggly, "{g^{*}\widetilde{r}}", from=1-1, to=1-3]
	\arrow[squiggly, "{\varepsilon_{(k')^{*}e^{*}\beta}}", from=1-3, to=1-5]
	\arrow[squiggly,"{u^{*}\eta_\beta}", from=1-5, to=1-7]
	\arrow[squiggly,"{\Psi(r)}"{description}, curve={height=24pt}, from=1-1, to=1-7]
	\arrow[squiggly, "{\widetilde{r}^\sharp}", curve={height=-24pt}, from=1-1, to=1-5]
\end{tikzcd}\]
Now, for $m \colon g^*\sigma \vertarr \coprod_f \beta$, taking $r:=\Phi(m)$, we get $\widetilde{r} = \widetilde{\Phi(m)} = \overline{m}^\flat$.
But then
\[\Psi(\Phi(m)) = u^* \eta_\beta \circ \widetilde{\Phi(m)}^\sharp =  u^* \eta_\beta \circ (\overline{m}^\flat)^\sharp = u^* \eta_\beta \circ \overline{m} = m.\]

In sum, we have established bijections\footnote{One can show that these bijections are ``natural in $\beta \in \mE_B$ and \emph{$(\cF,\coprod)$-quantifier-free} $\sigma \in \mE_S$,'' but seemingly naturality fails with respect to arbitrary $\sigma$.}
\[ \Phi_{\sigma,\beta} : \hom_{\mE_A}\Big(g^*\sigma,\coprod_f(\beta)\Big) \to \hom_{\mE_S}\Big(\sigma,\coprod_h \prod_{g'}e^*(\beta)\Big). \]

\noindent
\textbf{4) Non-quantifier-free case.}
What about general elements $\sigma \in \mE_S$? Consider an arbitrary $\sigma \in \mE_S$. Since $\mp$ has enough $(\coprod,\cF)$-quantifier-free elements, there exists a \emph{covering} $(\coprod,\cF)$-quantifier-free element $\widehat{\sigma}$ for $\sigma$, \ie, there exist an object $\widehat{S} \in \mB$, a morphism $s \colon \widehat{S} \fibarr S$ in $\cF$, and an object $\widehat{\sigma} \in \mE_{\widehat{S}}$ such that $\sigma \cong \coprod_s(\widehat{\sigma})$.

We want to lift our previous proof to the general case.

Consider the induced diagram, where all the squares involved are pullbacks:
\[\begin{tikzcd}
	{\widehat{B}} && {\widehat{Z}} && {\widehat{X}} \\
	& B && Z && X \\
	&& {\widehat{A}} && {\widehat{S}} \\
	&&& A && S
	\arrow[two heads, "{\widehat{h'}}"{description, pos=0.75}, from=1-3, to=3-3]
	\arrow[two heads, "\widehat{g}"{pos=0.8}, from=3-3, to=3-5]
	\arrow[two heads, "{\widehat{g'}}"{pos=0.4}, from=1-3, to=1-5]
	\arrow[two heads, "{\widehat{h}}"{description, pos=0.75}, from=1-5, to=3-5]
	\arrow[two heads, "g"{description}{pos=0.8}, from=4-4, to=4-6]
	\arrow[two heads, "h"{description, pos=0.45}, from=2-6, to=4-6]
	\arrow["{\widehat{e}}"{description, pos=0.6}, from=1-3, to=1-1]
	\arrow[two heads, "\widehat{f}"', curve={height=18pt}, from=1-1, to=3-3]
	\arrow[two heads, "{r''}", from=1-1, to=2-2]
	\arrow[two heads, "{s'}", from=3-3, to=4-4]
	\arrow[two heads, "r", from=1-5, to=2-6]
	\arrow[two heads, "s", from=3-5, to=4-6]
	\arrow["\lrcorner"{anchor=center, pos=0.125}, shift right=2, draw=none, from=1-1, to=4-4]
	\arrow["\lrcorner"{anchor=center, pos=0.125}, shift right=2, draw=none, from=2-4, to=4-6]
	\arrow["\lrcorner"{anchor=center, pos=0.125, rotate=45}, draw=none, from=3-3, to=4-6]
	\arrow["\lrcorner"{anchor=center, pos=0.125, rotate=-45}, draw=none, from=1-5, to=4-6]
	\arrow["\lrcorner"{anchor=center, pos=0.125}, shift right=3, draw=none, from=1-3, to=4-6]
	\arrow["\lrcorner"{anchor=center, pos=0.125}, shift left=3, draw=none, from=1-3, to=4-6]
 	\arrow[two heads, "{h'}"{description, pos=0.45}, from=2-4, to=4-4, crossing over]
  	\arrow[two heads, "{r'}"'{pos=0.2
   }, from=1-3, to=2-4]
   	\arrow[two heads, "{g'}"{pos=0.2}, from=2-4, to=2-6, crossing over]
    \arrow["e"{description, pos=0.7}, from=2-4, to=2-2, crossing over]
    \arrow[two heads, "f"', curve={height=24pt}, from=2-2, to=4-4, crossing over]
\end{tikzcd}\]

Via adjointness and Beck--Chevalley conditions we can reduce the general case to the quantifier-free case as follows:

 \begin{prooftree}
        \AxiomC{$g^*\sigma \stackrel{\text{(def.)}}{\cong} g^*(\coprod_s \widehat{\sigma}) \stackrel{\text{(BCC)}}{\cong}  \coprod_{s'} (\widehat{g})^*(\widehat{\sigma}) \rightsquigarrow \coprod_f(\beta)$}
        \RightLabel{(adj.)}
        \UnaryInfC{$\widehat{g}^* \widehat{\sigma}  \rightsquigarrow (s')^* \coprod_f(\beta) \stackrel{\text{(BCC)}}{\cong} \coprod_{\widehat{f}} \underbrace{(r'')^*(\beta)}_{=: \widehat{\beta}}$}
        \RightLabel{(Skolem)}
        \UnaryInfC{$\widehat{\sigma} \rightsquigarrow \coprod_{\widehat{h}} \prod_{\widehat{g'}} (\widehat{e})^*(\widehat{\beta})$}
 \end{prooftree}

On the other hand, we also find:
\begin{prooftree}
        \AxiomC{$\sigma \stackrel{\text{(def.)}}{\cong} \coprod_s \widehat{\sigma} \rightsquigarrow \coprod_h \prod_{g'} e^* \beta$}
        \RightLabel{(adj.)}
        \UnaryInfC{$\widehat{\sigma} \rightsquigarrow (s^*\coprod_h) \prod_{g'} e^* \beta \stackrel{\text{(BCC)}}{\cong} \coprod_{\widehat{h}} (r^* \prod_{g'}) e^*(\beta) \stackrel{\text{(BCC)}}{\cong}  \coprod_{\widehat{h}} (\prod_{\widehat{g'}} (r')^*) e^* \beta \stackrel{\text{(BCC)}}{\cong}
        \coprod_{\widehat{h}}  \prod_{\widehat{g'}} (\widehat{e})^* \underbrace{(r'')^*(\beta)}_{\stackrel{\text{(def.)}}{=} \widehat{\beta}}
        $}
 \end{prooftree}

\end{proof}

\begin{remark}\label{rem:loc-vs-glob-iso}
  Notice that the previous proof allows us to demonstrate the validity of Skolemisation only in the \emph{local} case. This result highlights and corrects an inaccuracy in the proof of the corresponding result, claiming the validity of the iso in the global case, presented in \cite[Proposition~2.8]{trotta_et_al:LIPIcs.MFCS.2021.87} for the non-dependent case. It is natural to question whether this result extends to the global case. The main challenge in obtaining a global \emph{natural isomorphism} lies in the fact that the current notion of having enough quantifier-free elements does not provide a \emph{canonical choice of witnesses}. Specifically, given an element of a fibre, there may be multiple quantifier-free elements representing that element. A potential solution could involve the imposition of the existence of a canonical element, accompanied by appropriate coherence conditions relating these canonical elements.
\end{remark}

\begin{corollary}
        In the case of posetal fibrations, the previous ``local'' isomorphisms do assemble to a natural (``global'') isomorphism.
\end{corollary}

\section{Dependent Gödel fibrations}

We can now naturally generalise the notion of Gödel fibration~\cite[Definition~2.9]{trotta_et_al:LIPIcs.MFCS.2021.87} to the dependent case.

\begin{definition}[dependent Gödel fibration]\label{def:goedel-fib}
    Let $\pair{\mB}{\cF}$ be a display map category and $\fibration{\mE}{\mp}{\mE}$ a dependent Skolem fibration over it. It is called a \emph{dependent Gödel fibration} if the full subfibration $\fibration{\mE'}{\mp'}{\mB}$ of $(\cF,\coprod)$-quantifier-free objects has enough $(\cF,\prod)$-quantifier-free objects.
\end{definition}

We shall henceforth drop the attribute ``dependent'' when referring to Skolem or Gödel fibrations in this generalised sense.

We show that dependent Gödel fibrations admit \emph{prenex normal forms}, generalising the existence of a formula $\beta = \beta(x,y,i)$ for each formula $\alpha = \alpha(i)$ such that
\[ \alpha(i) \equiv \exists x \forall y \beta(x,y,i). \]
The statement and proof are generalised from~\cite[Proposition~2.11]{trotta_et_al:LIPIcs.MFCS.2021.87}.

Note that, as in \cite[Remark~2.10]{trotta_et_al:LIPIcs.MFCS.2021.87}, $(\cF,\prod)$-quantifier-freeness in the fibration $\mp'$ does not necessarily entail $(\cF,\prod)$-quantifier-freeness in the fibration $\mp$, for a dependent Gödel fibration $\mp$.
\begin{example}
    By definition, every (dependent) Hilbert $\epsilon$- and $\tau$-fibration with fibred products and coproducts along $\cF$ (and whose base category has dependent products along $\cF$) is a Gödel fibration.
\end{example}
\subsection{Prenexation}

As in~\cite{trotta_et_al:LIPIcs.MFCS.2021.87}, Gödel fibrations do admit a kind of prenex normal form.

\begin{proposition}[Prenexation]
Let $\fibration{\mE}{\mp}{\mB}$ be a dependent Gödel fibration over a display map category $\pair{\mB}{\cF}$ with units. Then, for every object $A \in \mB$ and every $\alpha \in \mE_A$ there exist display maps an arrow $f : C \fibarr A$ and $g : B \fibarr C$ together with an element $\beta \in \mE_B$ such that
\[ \alpha \cong \coprod_f \prod_g \beta, \]
and $\beta$ is $(\cF,\prod)$-quantifier-free in the subfibration $\mp'$ of $(\cF,\coprod)$-quantifier-free elements of $\mp$.
 \end{proposition}

\begin{proof}
    Let $\alpha \in \mE_A$. Since $\mp$ is a dependent Gödel fibration, it is a Skolem fibration. Thus, $\mp$ has enough $(\cF,\coprod)$-quantifier-free objects, and in particular there exist an arrow $f : C \fibarr A$ and an element $\gamma \in \mE_\gamma$ with $\alpha \cong \coprod_f \gamma$. Then, since the full subfibration $\mp'$ of $(\cF,\coprod)$-quantifier-free elements has enough $(\cF,\prod)$-quantifier-free elements, there exists a further map $g : B \to C$ together with a $(\cF,\prod)$-quantifier-free element $\beta$ of $\mp'$ in $\mE_B$ such that $\gamma \cong \prod_g \beta$, hence $\alpha \cong \coprod_f \prod_g \beta$, as desired.
\end{proof}

\begin{remark}
    Again, these isomorphisms between individual hom-sets do not in general assemble to give a natural isomorphism between functors because of the lack of naturality of choice of covering quantifier-free elements, see~\Cref{rem:loc-vs-glob-iso}.
\end{remark}

\subsection{Characterisation as Dialectica fibrations}

We will now prove that Dialectica fibrations are the same as Gödel fibrations (up to fibred equivalence). Like in~\cite{trotta_et_al:LIPIcs.MFCS.2021.87}, we show this in a modular fashion that makes use of a deeper analysis of the coproduct completion, and the definition of the product completion in terms of the coproduct completion and the opposite of a fibration.

The roadmap is as follows. We will first establish a few technical results that finally enable us to show that the coproduct completion has enough $\coprod$-quantifier-free-elements. Furthermore, we show that a fibration is an instance of a coproduct completion if and only it has enough $\coprod$-quantifier-free-elements. Putting all the results together will then yield the main theorem, that a fibration is Gödel if and only if it is (fibred isomorphic to) the Dialectica construction of some fibration. Moreover, one can exhibit this fibration as the full subfibration of $\prod$-quantifier-free elements. 

In the following, we will always consider a fixed display map category $\pair{\mB}{\cF}$ with units and $\cF$-coproducts. We will not explicitly mention these conditions anymore.

We furst give a description of the unit of the adjunction of the cocartesian and cartesian transport along display maps in the coproduct completion of a fibration.

\begin{proposition}\label{prop:sigma-compl-unit}
 Let $\fibration{\mE}{\mp}{\mB}$ be a fibration over a display map category $\pair{\mB}{\cF}$ with units and $\cF$-dependent coproducts. Let us consider the $(\cF,\coprod)$-completion $\fibration{\Sigma_\cF(\mE)}{\Sigma_\cF(\mp)}{B}$. Consider an arrow $u : A´ \fibarr J$ in $\cF$.
  For $A \in \mB$, let $\sigma = (A, g : B \fibarr A, \beta \in \mE_B) \in (\Sigma_\cF(\mE))_A$. Then the unit $\kappa : \id \Rightarrow u^* \coprod_u$ of the adjunction $\coprod_u \dashv u^*$ of $\Sigma_\cF(\mp)$ at $\sigma$ is given by
 \[ \kappa_\sigma = \big( (g,\id_B) : g \to u^*(ug), \phi_{g,\beta} : \beta \to (u')^* \beta \big),\]
 where $\phi'$ is the $\mp$-cartesian filler as given in:
\[\begin{tikzcd}
	\beta \\
	& {u'^*\beta} & \beta \\
	B \\
	& {u^*B} & B
	\arrow["{\phi_{g,\beta}}"', dashed, from=1-1, to=2-2]
	\arrow[Rightarrow, no head, from=1-1, to=2-3]
	\arrow[Rightarrow, dotted, crossing over, no head, from=1-1, to=3-1]
	\arrow[from=2-2, to=2-3, cart]
	\arrow[Rightarrow, dotted, no head, from=2-2, to=4-2]
	\arrow[Rightarrow, dotted, no head, from=2-3, to=4-3]
	\arrow["{(g,\id_B)}"', from=3-1, to=4-2]
	\arrow[Rightarrow, no head, from=3-1, to=4-3]
	\arrow["{u'}"', from=4-2, to=4-3]
\end{tikzcd}\]

\end{proposition}

\begin{proof}
    We recall the following: let $\pair{\mB}{\cF}$ be a display map category with units and coproducts. Let $\fibration{\mathsf{F}}{q}{\mB}$ be a fibration  and $u : B \fibarr J$ be an arrow in $\cF$. Consider the adjunction $\coprod_u \dashv u^*$. Let $\alpha \in \mathsf{F}_A$ and $\beta \in \mathsf{F}_B$. The transpose of a vertical map $\psi:\coprod_u \alpha \vertarr \beta$ is given by the vertical map $\psi'$ arising as the unique filler to the cartesian map as below:
\[\begin{tikzcd}
	\alpha && {\coprod_u\alpha} \\
	{u^*\beta} && \beta
	\arrow[from=1-1, to=1-3, cocart]
	\arrow["{\psi'}"', dashed, from=1-1, to=2-1, squiggly]
	\arrow[from=2-1, to=2-3, cart]
	\arrow["\psi", from=1-3, to=2-3, squiggly]
	\arrow[from=1-1, to=2-3]
\end{tikzcd}\]

Hence, in the case of the $(\cF,\coprod)$-completion, the transposing map for this adjunction is given, for $\sigma = (A, g: B \fibarr A, \beta \in \mE_B)$ and $\tau = (J, h:C \fibarr J, \gamma \in \mE_C)$, by the map $\Psi_{\sigma,\tau} \colon \hom_{(\Sigma_\cF E)_{J}}(\coprod_u \sigma, \tau) \to \hom_{(\Sigma_\cF E)_{A}}(\sigma, u^*\tau)$, which maps a pair $(k,\phi)$ to $(k',\phi')$ as indicated in:
\[\begin{tikzcd}
	{\beta} && \gamma & \beta && {u'^*\gamma} \\
	B && C & B && {u^*C} \\
	& J &&& A
	\arrow["k", from=2-1, to=2-3]
	\arrow["\phi", from=1-1, to=1-3]
	\arrow["ug"', two heads, from=2-1, to=3-2]
	\arrow["h", two heads, from=2-3, to=3-2]
	\arrow["{k'}", from=2-4, to=2-6]
	\arrow["g"', two heads, from=2-4, to=3-5]
	\arrow["{u^*h}", two heads, from=2-6, to=3-5]
	\arrow["{\phi'}", from=1-4, to=1-6]
	\arrow["{\Psi_{\sigma,\tau}}", curve={height=-12pt}, maps to, from=2-3, to=2-4]
\end{tikzcd}\]

Here, $k' = (g,k):B \to u^*C = A \times_J C$ is the cartesian gap map:
\[\begin{tikzcd}
	B &&& B \\
	& {u^*C} &&& C \\
	A &&& J
	\arrow[Rightarrow, no head, from=1-1, to=1-4]
	\arrow["{(g,k)}"{description}, dashed, from=1-1, to=2-2]
	\arrow["g"', two heads, from=1-1, to=3-1]
	\arrow[two heads, from=2-2, to=3-1, "h'"]
	\arrow["k"{description}, from=1-4, to=2-5]
	\arrow[two heads, "h", from=2-5, to=3-4]
	\arrow[two heads, "ug"{pos=0.3}, from=1-4, to=3-4, swap]
	\arrow["u"{description}, two heads, from=3-1, to=3-4]
    \arrow["u'"{description, pos=0.4}, from=2-2, to=2-5, crossing over, two heads]
	\arrow["\lrcorner"{anchor=center, pos=0.125}, draw=none, from=2-2, to=3-4]
\end{tikzcd}\]
Furthermore, $\phi':\beta \to u'^*\gamma$ is the filler in the fibred square lying over the top square of the above diagram:
\[\begin{tikzcd}
	\beta && {\beta} \\
	{u'^*\gamma} && \gamma
	\arrow[from=1-1, to=1-3, Rightarrow, no head]
	\arrow["{\phi'}"', dashed, from=1-1, to=2-1]
	\arrow[from=2-1, to=2-3, cart]
	\arrow["\phi", from=1-3, to=2-3]
	\arrow[from=1-1, to=2-3]
\end{tikzcd}\]
Accordingly, the unit at $\sigma$ is given by
\[\kappa_\sigma := \Psi_{\sigma,\coprod_u \sigma}(\id_{\coprod_u \sigma}) = \big((g,\id_B): g \to u^*(ug), \phi': \beta \to u'^* \beta\big)\]
with $u' := (ug)^*u$, and $\phi' := \phi_{g,\beta}$ being the unique section of the cartesian lift $\mp^*(u',\beta) : (u')^*\beta \cartarr \beta$ such that with $\mp(\phi') = (g,\id_B)$.
\end{proof}

We now show that elements given by identity arrows together with some element in the fibre are $(\cF,\coprod)$-quantifier free elements in $\Sigma_\cF(\mp)$.

\begin{proposition}\label{prop:qf-free-elts-of-sigma}
    Let $\fibration{\mE}{\mp}{\mB}$ be a fibration over a display map category $\pair{\mB}{\cF}$ with units and $\cF$-dependent coproducts. Let us consider the $(\cF,\coprod)$-completion $\fibration{\Sigma_\cF(\mE)}{\Sigma_\cF(\mp)}{B}$. Let $I$ be an object of $\mB$ and $\alpha$ be an object of its fibre $\mE_I$. Then the object
    \[I_\alpha := (\id_I:I\tofib I, \alpha \in \mE_I)\] in the fibre $\big(\Sigma_\cF(\mE)\big)_I$ is an $(\cF,\coprod)$-quantifier-free element of $\Sigma_\cF(\mp)$.
\end{proposition}

\begin{proof}
We have to show that for all $f:J \fibarr I$ in $\cF$, the element $f^*I_\alpha 
= \big(J, \id_J, f^*(\alpha) \big) \in ({\Sigma_\cF} \mE)_J$ is an $(\cF,\coprod)$-quantifier-free splitting. For this, we are to show the following: given any $\sigma = (A,g:B \fibarr A, \beta \in \mE_B) \in \big(\Sigma_\cF (\mE)\big)_A$ together with $u:A \fibarr J$ and $\Phi \in \hom_{(\Sigma_\cF E)_{J}}(f^*I_\alpha, \coprod_u \sigma)$, there uniquely exist a section $s:J \to A$ of $u$ and a vertical arrow $\overline{\Phi}: f^*I_\alpha \vertarr s^*\sigma$ such that $\Phi = s^*\kappa_\sigma \circ \overline{\Phi}$, where $\kappa : \id_{(\Sigma \mE)_A} \Rightarrow u^*\coprod_u$ is the unit of the adjunction $u^* \vdash \coprod_u : \big(\Sigma_\cF (\mE)\big)_A \to \big(\Sigma_\cF (\mE)\big)_J$.

By construction of the $(\cF,\coprod)$-completion, any vertical morphism $\Phi$ as above takes the form of $\Phi = \big( (\id_J, r), \phi\big)$ for some $r:J \to B$ in $\mB$ (not necessarily in $\cF$) and $\phi:f^*\alpha \to \beta$ with $\mp(\phi) = r$, where furthermore $u(gr) = \id_J$. Hence, we set $s:=gr:J \to A$ as the candidate section of $u$.
 
The desired factorization demands $\Phi = s^*\kappa_\sigma \circ \overline{\Phi}$, vertically over $J$. Note that, using \Cref{prop:sigma-compl-unit}, $s^*$ acts on $\kappa_\sigma$ as follows:
\[\begin{tikzcd}
	{(g^*s)^*\beta} && \beta \\
	& {(s')^*(u')^*\beta} && {(u')^*\beta} \\
	{s^*B} && B \\
	& B && {u^*B} \\
	& J && A
	\arrow[from=1-1, to=1-3, cart]
	\arrow["{\phi''}", dashed, from=1-1, to=2-2]
	\arrow[Rightarrow, dotted, no head, from=1-1, to=3-1]
	\arrow["{\phi'}", from=1-3, to=2-4]
	\arrow[Rightarrow, dotted, no head, from=1-3, to=3-3]
	\arrow[from=2-2, to=2-4, cart]
	\arrow[Rightarrow, dotted, no head, from=2-2, to=4-2]
	\arrow[Rightarrow, dotted, no head, from=2-4, to=4-4]
	\arrow["{s''}", dashed, from=3-1, to=4-2]
	\arrow["{s^*g}"{description}, curve={height=12pt}, from=3-1, to=5-2, two heads]
	\arrow["\lrcorner"{anchor=center, pos=0.125}, shift right=4, draw=none, from=3-1, to=5-4]
	\arrow["{(g,B)}", from=3-3, to=4-4]
	\arrow["g"{description, pos=0.2}, curve={height=12pt}, from=3-3, to=5-4, two heads]
	\arrow["ug", two heads, from=4-2, to=5-2]
	\arrow["\lrcorner"{anchor=center, pos=0.125}, draw=none, from=4-2, to=5-4]
	\arrow["{u^*(ug)}", two heads, from=4-4, to=5-4]
	\arrow["s"{description, pos=0.4}, from=5-2, to=5-4]
 	\arrow["{s'}"{description, pos=0.4}, from=4-2, to=4-4, crossing over]
  	\arrow["{g^*s}", from=3-1, to=3-3, crossing over]
\end{tikzcd}\]
But to the right we have the following pasted squares:
\[\begin{tikzcd}
	{(g^*s)^*\beta} && \beta \\
	& {(s')^*(u')^*\beta} && {(u')^*\beta} && \beta \\
	{s^*B} && B \\
	& B && {u^*B} && B \\
	& J && A && J
	\arrow[from=1-1, to=1-3, cart]
	\arrow["{\phi''}", dashed, from=1-1, to=2-2]
	\arrow[Rightarrow, dotted, no head, from=1-1, to=3-1]
	\arrow["{\phi'}", dashed, from=1-3, to=2-4]
	\arrow[curve={height=-12pt}, Rightarrow, no head, from=1-3, to=2-6]
	\arrow[Rightarrow, dotted, no head, from=1-3, to=3-3]
	\arrow[from=2-2, to=2-4, cart]
	\arrow[Rightarrow, dotted, no head, from=2-2, to=4-2]
	\arrow[from=2-4, to=2-6, cart]
	\arrow[Rightarrow, dotted, no head, from=2-4, to=4-4]
	\arrow[Rightarrow, dotted, no head, from=2-6, to=4-6]
	\arrow["{g^*s}", from=3-1, to=3-3]
	\arrow["{s''}", dashed, from=3-1, to=4-2]
	\arrow["{s^*g}"{description}, curve={height=12pt}, two heads, from=3-1, to=5-2]
	\arrow["\lrcorner"{anchor=center, pos=0.125}, shift right=4, draw=none, from=3-1, to=5-4]
	\arrow["{(g,B)}", from=3-3, to=4-4]
	\arrow[curve={height=-12pt}, Rightarrow, no head, from=3-3, to=4-6]
	\arrow["g"{description, pos=0.2}, curve={height=12pt}, two heads, from=3-3, to=5-4]
	\arrow["ug", two heads, from=4-2, to=5-2]
	\arrow["\lrcorner"{anchor=center, pos=0.125}, draw=none, from=4-2, to=5-4]
	\arrow["{u'}", from=4-4, to=4-6, two heads]
	\arrow["{u^*(ug)}", two heads, from=4-4, to=5-4]
	\arrow["\lrcorner"{anchor=center, pos=0.125}, draw=none, from=4-4, to=5-6]
	\arrow["ug", from=4-6, to=5-6, two heads]
	\arrow["s"{description, pos=0.4}, from=5-2, to=5-4]
	\arrow["u", from=5-4, to=5-6, two heads]
 \arrow["{s'}"{description, pos=0.4}, two heads, from=4-2, to=4-4, crossing over]
\end{tikzcd}\]
Since $us=\id_J$, we get $u's' \cong \id_B$. Now, as identities pull back to identities, composing yields (up to isomorphism) the following pasted diagrams:
\[\begin{tikzcd}
	{(g^*s)^*\beta} \\
	& \beta && \beta \\
	{s^*B} \\
	& { B} && B \\
	& J && J
	\arrow["{\phi''}"', dashed, from=1-1, to=2-2]
	\arrow[from=1-1, to=2-4, cart]
	\arrow[Rightarrow, dotted, no head, from=1-1, to=3-1]
	\arrow[Rightarrow, no head, from=2-2, to=2-4]
	\arrow[Rightarrow, dotted, no head, from=2-2, to=4-2]
	\arrow[Rightarrow, dotted, no head, from=2-4, to=4-4]
	\arrow["{s''}", dashed, from=3-1, to=4-2]
	\arrow["{g^*s}", curve={height=-12pt}, from=3-1, to=4-4]
	\arrow["{s^*g}"', curve={height=12pt}, from=3-1, to=5-2, two heads]
	\arrow[Rightarrow, no head, from=4-2, to=4-4]
	\arrow["ug"', two heads, from=4-2, to=5-2]
	\arrow["\lrcorner"{anchor=center, pos=0.125}, draw=none, from=4-2, to=5-4]
	\arrow["ug", two heads, from=4-4, to=5-4]
	\arrow[Rightarrow, no head, from=5-2, to=5-4]
\end{tikzcd}\]
But this means $s'' = g^*s$ and $\phi'' = \mp^*(g^*s, \beta)$.

We claim that $\overline{\Phi} = \big((\id_J,\overline{r}), \overline{\phi}: f^*\alpha \to (s')^*\beta\big)$ is an appropriate candidate, where the maps arise as follows, starting from the pullback in the middle:
\[\begin{tikzcd}
	{f^*\alpha} & {(g^*s)^*\beta} & \beta \\
	J & {s^*B} & B & B \\
	J & J & A & J
	\arrow["{\overline{\phi}}", dashed, from=1-1, to=1-2]
	\arrow["\phi"{pos=0.3}, curve={height=-24pt}, from=1-1, to=1-3]
	\arrow[Rightarrow, dotted, no head, from=1-1, to=2-1]
	\arrow[from=1-2, to=1-3, cart]
	\arrow[Rightarrow, dotted, no head, from=1-2, to=2-2]
	\arrow[Rightarrow, dotted, no head, from=1-3, to=2-3]
	\arrow["{{\overline{r}}}", dashed, from=2-1, to=2-2]
	\arrow["r"{description, pos=0.3}, curve={height=-24pt}, from=2-1, to=2-3]
	\arrow[Rightarrow, no head, from=2-1, to=3-1]
	\arrow["{{g^*s}}", from=2-2, to=2-3]
	\arrow["{{s^*g}}"', two heads, from=2-2, to=3-2]
	\arrow["\lrcorner"{anchor=center, pos=0.125}, draw=none, from=2-2, to=3-3]
	\arrow[Rightarrow, no head, from=2-3, to=2-4]
	\arrow["g", two heads, from=2-3, to=3-3]
	\arrow["ug", two heads, from=2-4, to=3-4]
	\arrow[Rightarrow, no head, from=3-1, to=3-2]
	\arrow["s", from=3-2, to=3-3]
	\arrow[curve={height=18pt}, Rightarrow, no head, from=3-2, to=3-4]
	\arrow["u", from=3-3, to=3-4]
\end{tikzcd}\]
As desired, this gives rise to the factorization:

\[\begin{tikzcd}
	{f^*\alpha} && \beta \\
	& {(g^*s)^*\beta} \\
	J && B \\
	J & {s^*B} & J \\
	& J
	\arrow["{\overline{r}}"{description}, dashed, from=3-1, to=4-2]
	\arrow["{g^*s}"{description}, from=4-2, to=3-3]
	\arrow[Rightarrow, no head, from=3-1, to=4-1]
	\arrow[Rightarrow, no head, from=4-1, to=5-2]
	\arrow["{s^*g}", two heads, from=4-2, to=5-2]
	\arrow["ug", two heads, two heads, from=3-3, to=4-3]
	\arrow[Rightarrow, no head, from=4-3, to=5-2]
	\arrow["r"{description, pos=0.7}, from=3-1, to=3-3]
	\arrow["\phi"{description}, from=1-1, to=1-3]
	\arrow["{\overline{\phi}}"{description}, dashed, from=1-1, to=2-2]
	\arrow[cart, from=2-2, to=1-3]
	\arrow[Rightarrow, dotted, no head, from=1-1, to=3-1]
	\arrow[Rightarrow, dotted, no head, from=2-2, to=4-2]
	\arrow[Rightarrow, dotted, no head, from=1-3, to=3-3]
	\arrow[Rightarrow, no head, from=4-1, to=4-2]
	\arrow[Rightarrow, no head, from=4-2, to=4-3]
\end{tikzcd}\]

As for uniqueness, assume we have another section $q:J \to A$ of $u$, together with a morphism $\Xi = ( m: \id_B \to q^*g, \psi:f^*\alpha \to q^*\beta)$. The action of $q^*$ on the unit gives rise to the induced morphism in:
\[\begin{tikzcd}
	{q^*B} &&& B \\
	& B &&& {A \times_JB} \\
	J &&& A
	\arrow[from=1-1, to=1-4]
	\arrow["{\pi}"{description}, dashed, from=1-1, to=2-2]
	\arrow["{q^*g}"', two heads, from=1-1, to=3-1]
	\arrow["ug"{description}, two heads, from=2-2, to=3-1]
	\arrow["{(g,\id_B)}"{description}, from=1-4, to=2-5]
	\arrow["{u^*(ug)}"{description}, two heads, from=2-5, to=3-4]
	\arrow["g"{pos=0.2}, two heads, from=1-4, to=3-4, swap]
	\arrow["q"{description}, from=3-1, to=3-4]
 \arrow[from=2-2, to=2-5,crossing over]
	\arrow["\lrcorner"{anchor=center, pos=0.125}, draw=none, from=2-2, to=3-4]
	\arrow["\lrcorner"{anchor=center, pos=0.125}, shift left=2, draw=none, from=1-1, to=3-4]
\end{tikzcd}\]
The factorization condition $\Phi = q^*\kappa_\sigma \circ \xi$ then in particular entails the factorization:
\[\begin{tikzcd}
	J && {q^*B} && B \\
	J && J && J
	\arrow["m", from=1-1, to=1-3]
	\arrow["\pi", from=1-3, to=1-5]
	\arrow[Rightarrow, no head, from=1-1, to=2-1]
	\arrow[Rightarrow, no head, from=2-1, to=2-3]
	\arrow[Rightarrow, no head, from=2-3, to=2-5]
	\arrow["ug", from=1-5, to=2-5, two heads]
	\arrow["r"{description}, curve={height=-24pt}, from=1-1, to=1-5]
	\arrow["{q^*g}"', from=1-3, to=2-3,  two heads,]
\end{tikzcd}\]
But then, the map $m$ necessarily occurs as the gap map $(r,\id_J)$ in:
\[\begin{tikzcd}
	J \\
	&& {B \times_A J} && B \\
	&& J && A
	\arrow["{q^*g}"', from=2-3, to=3-3, two heads]
	\arrow["q"', from=3-3, to=3-5]
	\arrow["\pi", from=2-3, to=2-5]
	\arrow["g", from=2-5, to=3-5, two heads]
	\arrow["r"{description}, curve={height=-18pt}, from=1-1, to=2-5]
	\arrow[curve={height=12pt}, Rightarrow, no head, from=1-1, to=3-3]
	\arrow["\lrcorner"{anchor=center, pos=0.125}, draw=none, from=2-3, to=3-5]
	\arrow["m"{description}, dashed, from=1-1, to=2-3]
\end{tikzcd}\]
So $q = gr = s$, and accordingly $m=r'$, as above, and $\xi = \phi'$, \ie, $\Xi = \overline{\Phi}$. 
\end{proof}

\begin{proposition}\label{rem:fact}
    Let $\fibration{\mE}{\mp}{\mB}$ be a fibration over a display map category $\pair{\mB}{\cF}$. Then any vertical morphism $(h,\phi) \colon (I,f\colon A \fibarr I,\alpha) \to (I,g \colon B \fibarr I,\beta)$ in the completion $\Sigma_\cF(\mp)$ factors uniquely, up to isomorphism, as $(I,h,\phi) = \varepsilon_{(I,g,\beta)} \circ \coprod_f(h',\phi')$, where $h' = (\id_A,h) \colon A \to A \times_I B$, and $\phi'$ is the unique filler such that $\phi = \mp^*(g^*f, (g^*f)^*\beta \to \beta) \circ \phi'$.
\end{proposition}

\begin{proof}
    Let $\fibration{\mE}{\mp}{\mB}$ be a fibration over a display map category $\pair{\mB}{\cF}$. Fix an object $I \in \mB$, and in $\big(\Sigma_\cF(\mp)\big)_I$ consider an arrow $(h,\phi) \colon (I,f \colon A \fibarr I, \alpha \in \mE_A) \to (I, g \colon B \fibarr I, \beta \in \mE_B)$, \emph{i.e.}, $h \colon A \to B$ with $g \circ h = f$, and $\phi \colon \alpha \to \beta$ in $\mE$ with $\mp(\phi) = h$. We have $(I,f,\alpha) = \coprod_f(A_\alpha)$, so for $(h,\phi) \colon \coprod_f(A_\alpha) \to (I,g,\beta)$ we can consider its right adjoint transpose
    \[ (h',\phi') \colon A_\alpha \to f^*(I,g,\beta) = (A,f^*g \colon A \times_I B \fibarr A, (g^*f)^*\beta),\]
    where $h' = (\id_A, h)\colon A \to A \times_I B$ in
\[\begin{tikzcd}
	A \\
	& {A \times_I B} & B \\
	& A & I
	\arrow["{f^*g}", from=2-2, to=3-2, two heads]
	\arrow["f"', from=3-2, to=3-3, two heads]
	\arrow["{g^*f}", from=2-2, to=2-3, two heads]
	\arrow["g", from=2-3, to=3-3, two heads]
	\arrow["h", curve={height=-18pt}, from=1-1, to=2-3]
	\arrow[curve={height=12pt}, Rightarrow, no head, from=1-1, to=3-2]
	\arrow["{h'}"{description}, dashed, from=1-1, to=2-2]
	\arrow["\lrcorner"{anchor=center, pos=0.125}, draw=none, from=2-2, to=3-3]
\end{tikzcd}\]
    and $\phi' \colon \alpha \to (g^*f)^*\beta$ is the unique filler such that $\phi = \mp^*(g^*f, \beta) \circ \phi'$. This is the unique map making the following diagram commute, where $\varepsilon$ denotes the counit of the adjunction $\coprod_f \dashv f^*$ for $\Sigma_\cF(\mp)$:
\[\begin{tikzcd}
	{\coprod_f(A_\alpha)=(I,f,\alpha)} && {(I,g,\beta)} \\
	{\coprod_f(f^*(I,g,\beta))=(I,f \circ f^*g,(g^*f)^*\beta)}
	\arrow["{(I,h,\phi)}", from=1-1, to=1-3]
	\arrow["{\coprod_f(h',\phi')}"', from=1-1, to=2-1]
	\arrow["{\varepsilon_{(I,g,\beta)}}"', from=2-1, to=1-3]
\end{tikzcd}\]
Explicitly, this amounts to:
\[\begin{tikzcd}
	\alpha && {(g^*f)^*\beta} && \beta \\
	A && {A \times_IB} && B \\
	I && I && I
	\arrow["f"', from=2-1, to=3-1, two heads]
	\arrow["{h'}"{description}, from=2-1, to=2-3]
	\arrow[Rightarrow, no head, from=3-1, to=3-3]
	\arrow["{f \circ f^*g}", from=2-3, to=3-3, two heads]
	\arrow["{g^*f}"{description}, from=2-3, to=2-5, two heads]
	\arrow[Rightarrow, no head, from=3-3, to=3-5]
	\arrow["g", from=2-5, to=3-5, , two heads]
	\arrow["{\phi'}"{description}, from=1-1, to=1-3]
	\arrow["{\mathrm{cart}}"{description}, from=1-3, to=1-5]
	\arrow["\phi"{description}, curve={height=-24pt}, from=1-1, to=1-5]
	\arrow["h"{description}, curve={height=-18pt}, from=2-1, to=2-5]
\end{tikzcd}\]
Let $g' := f^*g$ and $\beta' := (g^*f)^*\beta$, then $g'h' = \id_A$. We write $A' := A \times_I B$, and by $(\cF, \coprod)$-quantifier-freeness of $A_\alpha$, the arrow
\[ (h',\id_A,\phi') \colon A_\alpha = (A,\id_A,\alpha) \to f^*(I,g,\beta) = (A,g',\beta') = \coprod_{g'} (A')_{\beta'}\]
factors uniquely as:
\[\begin{tikzcd}
	{A_\alpha } &&&& {(A,g',\beta') \cong \coprod_{g'} \big((A')_{\beta'}\big)} \\
	&& {A_{\beta'}\cong (h')^*\big((A')_{\beta'}\big)}
	\arrow["{(h',\id_A,\phi')}", from=1-1, to=1-5]
	\arrow["{(\id_A,\phi')}"', dashed, from=1-1, to=2-3]
	\arrow["{(h')^* \eta_{(A'_{\beta'})}}"', from=2-3, to=1-5]
\end{tikzcd}\]
Here, $\eta_{(A')_{\beta'}} \colon (A')_{\beta'} \to (g')^* \coprod_{g'} (A')_{\beta'} \cong (g'' \colon A'' \fibarr A', (g'')^*\beta')$ is the unit given by
\[ (\delta_{g'} \colon \id_{A'} \to g'', \sigma_{g',\beta'} \colon \beta' \to (g'')^* \beta')\]
where $\delta_{g'}$ is the diagonal of $g'$ in the sense of
\[\begin{tikzcd}
	{A'} \\
	& {A''} & {A'} \\
	& {A'} & A
	\arrow["{g''}", from=2-2, to=3-2, two heads]
	\arrow["{g'}"', from=3-2, to=3-3, two heads]
	\arrow["{g''}", from=2-2, to=2-3, two heads]
	\arrow["{g'}", from=2-3, to=3-3, two heads]
	\arrow[curve={height=-12pt}, Rightarrow, no head, from=1-1, to=2-3]
	\arrow[curve={height=6pt}, Rightarrow, no head, from=1-1, to=3-2]
	\arrow["{\delta_{g'}}", dashed, from=1-1, to=2-2]
	\arrow["\lrcorner"{anchor=center, pos=0.125}, draw=none, from=2-2, to=3-3]
\end{tikzcd}\]
and $\sigma_{g',\beta'}$ is the unique filler satisfying $\mp^*(g'',\beta') \circ \sigma_{g',\beta'} = \id_{\beta'}$, with $\mp(\sigma_{g',\beta'}) = g''$.
Cartesian reindexing by $h'$ then gives, up to isomorphism, the map
\[ (h')^*\eta_{(A')_{\beta'}} \colon (h')^*(A'_{\beta'}) \cong A_{\beta'} \to (h')^*(g')^*\coprod_{g'}(A'_{\beta'}) \cong (A,g',\beta')\]
defined by
\[ (h')^*\eta_{(A')_{\beta'}} = (h' \colon \id_A \to g', \sigma' \colon \beta \to \beta'). \]
Here, we have used that $(h')^*(\delta_{g'})$ can be identified with $h''$, as becomes transparent from the following diagram and the pullback lemma:
\[\begin{tikzcd}
	A && {A'} \\
	& {A'} && {A''} && {A'} \\
	A && {A'} && A
	\arrow["{h'}", from=1-1, to=1-3]
	\arrow["{\delta_{g'}}", from=1-3, to=2-4]
	\arrow["{g''}"{description}, from=2-4, to=2-6, two heads]
	\arrow["{h'}"{description}, dashed, from=1-1, to=2-2]
	\arrow[Rightarrow, no head, from=1-1, to=3-1]
	\arrow[from=2-2, to=3-1, two heads]
	\arrow["{h'}"{description}, from=3-1, to=3-3]
	\arrow["{g''}", from=2-4, to=3-3, two heads]
	\arrow["{g'}"{description}, from=3-3, to=3-5, two heads]
	\arrow["{g'}"{description}, from=2-6, to=3-5, two heads]
	\arrow[Rightarrow, no head, from=1-3, to=3-3]
	\arrow[curve={height=24pt}, Rightarrow, no head, from=3-1, to=3-5]
 	\arrow["{h''}"{description, pos=0.3}, from=2-2, to=2-4, crossing over]
	\arrow["\lrcorner"{anchor=center, pos=0.125}, draw=none, from=2-2, to=3-3]
	\arrow["\lrcorner"{anchor=center, pos=0.125}, shift left=3, draw=none, from=1-1, to=3-3]
	\arrow["\lrcorner"{anchor=center, pos=0.125}, draw=none, from=2-4, to=3-5]
	\arrow[curve={height=-12pt}, Rightarrow, no head, from=1-3, to=2-6]
\end{tikzcd}\]
The map $\sigma'$ is computed as the following filler, in a diagram lying in the $\mp$-fibre over the top square on the left hand side:
\[\begin{tikzcd}
	{\mathllap{\beta\cong}(h')^*\beta'} && {\beta'} \\
	{\mathllap{\beta'\cong} (h'')^*(g'')^*\beta'} && {(g'')^*\beta'}
	\arrow["{\mathrm{cart}}"{description}, from=1-1, to=1-3]
	\arrow[from=1-1, to=2-1]
	\arrow["{\mathrm{cart}}"{description}, from=2-1, to=2-3]
	\arrow["\sigma", from=1-3, to=2-3]
\end{tikzcd}\]
\end{proof}

Using the elements $I_\alpha$, we now show that $\Sigma_\cF(\mp)$ has enough $(\cF,\coprod)$-quantifier-free objects.

\begin{proposition}\label{prop:fact}
    Let $\fibration{\mE}{\mp}{\mB}$ over a display map category $\pair{\mB}{\cF}$. Then the $(\cF,\coprod)$-quantifier-free objects of $\Sigma_\cF(\mp)$ are, up to isomorphism, the elements $I_\alpha = (I, \id_I, \alpha)$. In particular, since every object $(I, f \colon B \fibarr I, \beta)$ of $\Sigma_\cF(\mp)$ satisfies
    \[ (I,f,\beta) \cong \coprod_f (B, \id_{B}, \beta), \]
    it is the case that $\Sigma_\cF(\mp)$ has enough $(\cF,\coprod)$-quantifier-free objects.
\end{proposition}

\begin{proof}
    From \Cref{prop:qf-free-elts-of-sigma}, we know that the elements of the form $I_\alpha$ are $(\cF, \coprod)$-quantifier-free.

    We assume $\Phi := (I, f \colon A \fibarr I, \alpha \in \mE_\alpha)$ is $(\cF, \coprod)$-quantifier-free. Then, we can factor the identity $\id_\Phi \colon (I,f,\alpha) \to (I,f,\alpha) = \coprod_f A_\alpha$, using a section $s \colon I \to A$ of $f$, as:
    \[\begin{tikzcd}
	\Phi && {\Phi = \coprod_f(A_\alpha)} \\
	& {s^*(A_\alpha)}
	\arrow["{\id_\Phi}", from=1-1, to=1-3]
	\arrow["\iota"', dashed, from=1-1, to=2-2]
	\arrow["{s^*\kappa_{A_\alpha}}"', from=2-2, to=1-3]
    \end{tikzcd}\]
    where $\kappa \colon \id \Rightarrow  f^* \coprod_f$ is the unit of the adjunction $\coprod_f \dashv f^*$ of $\Sigma_\cF(\mp)$. At $A_\alpha$, the unit is given by the pair
    \[ \kappa_{A_\alpha} = \Big(\delta_f \colon A \to A \times_I A, \eta_\alpha^f \colon \alpha \to f^*\coprod_f \alpha \Big),\]
    where $\delta_f \colon A \to A\times_I A$ is the diagonal of $f$, and $\eta^f$ is the unit of the adjunction $\coprod_f \dashv f^*$ of the fibration $\mp$. Reindexing by $s$ yields 
    \[s^*\kappa_{A_\alpha} = \Big( \rho_f \colon I \to I \times_A (A \times_I A), s^*\eta_\alpha^f \colon s^*\alpha \to \alpha \Big)\]
    where the map $\rho_f \colon I \to I \times_A (A \times_I A)$ arises from:
\[\begin{tikzcd}
	I &&& A \\
	& {I \times_A(A \times_IA)} &&& {A \times_I A} && A \\
	I &&& A && I
	\arrow[from=2-2, to=3-1, two heads]
	\arrow[Rightarrow, no head, from=1-4, to=3-4]
	\arrow["{\delta_f}", from=1-4, to=2-5]
	\arrow[from=2-5, to=3-4, two heads]
	\arrow["s"', from=3-1, to=3-4]
	\arrow["{\rho_f}", dashed, from=1-1, to=2-2]
	\arrow["s", from=1-1, to=1-4]
	\arrow[Rightarrow, no head, from=1-1, to=3-1]
	\arrow["\lrcorner"{anchor=center, pos=0.125}, draw=none, from=2-2, to=3-4]
	\arrow["\lrcorner"{anchor=center, pos=0.125}, shift left, draw=none, from=1-1, to=3-4]
	\arrow[from=2-5, to=2-7, two heads]
	\arrow["f"', from=2-7, to=3-6, two heads]
	\arrow["f"', from=3-4, to=3-6, two heads]
	\arrow["\lrcorner"{anchor=center, pos=0.125}, draw=none, from=2-5, to=3-6]
	\arrow[curve={height=-12pt}, Rightarrow, no head, from=1-4, to=2-7]
	\arrow[curve={height=24pt}, Rightarrow, no head, from=3-1, to=3-6, shorten >=10pt]
 	\arrow[from=2-2, to=2-5, crossing over]
\end{tikzcd}\]
Now, by pasting of the two front pullback squares, this yields up to isomorphism the diagram:
\[\begin{tikzcd}
	I \\
	& A && A \\
	& I && I
	\arrow["f"', from=2-2, to=3-2, two heads]
	\arrow[Rightarrow, no head, from=3-2, to=3-4]
	\arrow[Rightarrow, no head, from=2-2, to=2-4]
	\arrow["f", from=2-4, to=3-4, two heads]
	\arrow["s", curve={height=-18pt}, from=1-1, to=2-4]
	\arrow[curve={height=12pt}, Rightarrow, no head, from=1-1, to=3-2]
	\arrow["\lrcorner"{anchor=center, pos=0.125}, draw=none, from=2-2, to=3-4]
	\arrow["s"{description}, dashed, from=1-1, to=2-2]
\end{tikzcd}\]
This means, up to isomorphism, we can write
\[ s^*\kappa_{A_\alpha} = \Big( s \colon I \to A, s^*\eta_\alpha^f \colon s^*\alpha \to \alpha \Big).\]
By the factorization condition, we get:
\[\begin{tikzcd}
	\alpha && {s^*\alpha} && \alpha \\
	A && I && A \\
	I && I && I
	\arrow["f"{description}, from=2-1, to=2-3, two heads]
	\arrow["s"{description}, from=2-3, to=2-5]
	\arrow["{s^*\eta_\alpha^f}"{description}, from=1-3, to=1-5]
	\arrow[curve={height=-24pt}, Rightarrow, no head, from=1-1, to=1-5]
	\arrow["f"', from=2-1, to=3-1, two heads]
	\arrow[Rightarrow, no head, from=3-1, to=3-3]
	\arrow[Rightarrow, no head, from=3-3, to=3-5]
	\arrow["f", from=2-5, to=3-5, two heads]
	\arrow[Rightarrow, no head, from=2-3, to=3-3]
	\arrow[curve={height=-18pt}, Rightarrow, no head, from=2-1, to=2-5]
	\arrow["\psi"{description}, from=1-1, to=1-3]
\end{tikzcd}\]
But that means $s \circ f = \id_A$, so $s$ and $f$ are both isomorphisms. This means that their (co-)cartesian liftings are isomorphisms too. In particular, $s^*\eta^f_\alpha \cong \eta^f_\alpha$, and we obtain
\[s^*\kappa_{A_\alpha} = \Big(s \colon I \to A, \eta^f_\alpha \colon \alpha \to \coprod_f f^* \alpha \cong \alpha\Big).\]
By $2$-out-of-$3$, since both $\mp^*(f,\alpha)$ and $\mp_!(f,\alpha)$ are isomorphisms, then so is $\eta^f_\alpha$. Again, by $2$-out-of-$3$, so must be $\psi$. All in all, $\iota = ((f,\id_I),\psi)$ turns out to be an isomorphism $\Phi = (I,f,\alpha) \cong (I,\id_I,\alpha) = I_\alpha$.
\end{proof}

We now characterize the $\Sigma_\cF$-completions as excatly those cocomplete fibrations with enough $(\cF,\coprod)$-quantifier-free elements.

\begin{theorem}\label{thm:coprod-compl-enough-quant}
    A fibration $\fibration{\mE}{\mp}{\mB}$ with $\cF$-coproducts is an instance of an $\cF$-coproduct completion over $\mB$ (up to fibred equivalence) if and only if it has enough $(\cF,\coprod)$-quantifier-free elements.
\end{theorem}

\begin{proof}
    Let $\fibration{\overline{\mE}}{\overline{\mp}}{\mB}$ be the full subfibration of $\mp$ arising when restricting to the $(\cF,\coprod)$-quantifier-free objects of $\mp$. By the universal property of the $\cF$-coproduct completion, there exists a unique morphism of fibrations with $\cF$-coproducts such that the following diagram commutes:
\[\begin{tikzcd}
	{\overline{\mp}} & \mp \\
	{\sum_\cF(\overline{\mp})}
	\arrow["{\eta_{\overline{\mp}}}"', from=1-1, to=2-1]
	\arrow["\iota", hook, from=1-1, to=1-2]
	\arrow["F"', dashed, from=2-1, to=1-2]
\end{tikzcd}\]
    We will show that $F$ is an equivalence. We abbreviate $\mp' := \sum_\cF(\mp) \colon \mE' := \sum_\cF(\mE) \to \mB$. We will denote the cocartesian transports as $\sum_f := \coprod_f^\mp$ and $\coprod_f := \coprod_f^{\mp'} 
 =\coprod_f^{\sum_\cF(\overline{\mp})}$. We observe that, on the $(\cF,\coprod)$-quantifier-free objects, $F$ acts as the projection to the $\mE$-part of $\mE'$ in the sense that:
    \[ F(I_\gamma) = F(I, \id_I, \gamma \in \overline{E}_I) = (F \circ \eta_{\overline{\mp}})(\gamma) = \iota(\gamma) =
 \gamma.\]
    \emph{Essential surjectivity:} Let $\alpha \in \mE_I$. Since $\mp$ has enough $(\cF,\coprod)$-quantifier-free elements there exists $J \in \mB$, $f \colon J \to I$ in $\mB$ s.t.~$\sum_f \beta \cong \alpha$. Since $F$ preserves $\cF$-coproducts, we obtain
    \[ F(I,f,\beta) = F\Big(\coprod_f(J,\id_J,\beta)\Big) \cong \sum_f F(J,\id_J,\beta) = \sum_f \beta.\]
    as desired.
    
    \emph{Full faithfulness:} It suffices to show that $F \colon \mp' \to \mp$ gives rise to a family of equivalences $F_I \colon \mE'_I \to \mE_I$. We have shown essential surjectivity of the $F_I$ so it only remains to prove fully faithfulness. Recall from~\Cref{prop:fact} the factorization of a morphism $(h,\phi) \colon (I,f,\alpha) \to (I,g,\beta)$ (up to isomorphism) as $(I,h,\phi) = \varepsilon_{(I,g,\beta)} \circ \coprod_f (h',\phi')$, where $(h',\phi') \colon A_\alpha \to f^*(I,g,\beta)$ is the right adjoint transpose of $(I,h,\phi)$:

    \[\begin{tikzcd}
	{} & {\mathllap{(I,A,\alpha)=}\coprod_f(A_\alpha)} && {(I,g,\beta)} \\
	& {\coprod_ff^*(I,g,\beta)}
	\arrow["{(I,h,\phi)}", from=1-2, to=1-4]
	\arrow["{\coprod_f(h',\phi')}"', from=1-2, to=2-2]
	\arrow["{\varepsilon_{(I,g,\beta)}}"', from=2-2, to=1-4]
    \end{tikzcd}\]

    Here, $(h',\phi') \colon A_\alpha \to f^*(I,g,\beta)$ factors as follows:
\[\begin{tikzcd}
	{A_\alpha} && {f^*(I,g,\beta)} \\
	& {A_{\beta'} \cong (h')^*((A')_{\beta'})}
	\arrow["{(h',\phi')}", from=1-1, to=1-3]
	\arrow["{\iota_{\overline{\mp}}\phi'=(\id_A,\phi)}"', from=1-1, to=2-2]
	\arrow["{(h')^*\eta_{A'_{\beta'}}}"', from=2-2, to=1-3]
\end{tikzcd}\]
Taken together, we get the factorization
\[ (I,h,\phi) = \varepsilon_{I,g,\beta} \circ \coprod_f(h',\phi') = \varepsilon_{I,g,\beta} \circ \coprod_f\big( (h')^*\eta_{(A'_{\beta'})}\big) \circ \coprod_f \eta_{\overline{\mp}}(\phi'). \]
Since $F$ preserves $\cF$-sums and commutes with the inclusions, we obtain
\[ F(I,h,\phi) = \varepsilon_{\coprod_g \beta} \circ \Big( \sum_f \big( (h')^* \eta_{A'_{\beta'}} \big) \Big) 
\circ \Big( \sum_f \phi' \Big)\]
which is indeed an arrow $\sum_f \alpha \to \sum_f \beta$. Analogously, every arrow $\sum_f \alpha \to \sum_f \beta$ in $\mE'_I$ can be uniquely factored as such a composition, using the same arguments, and by full faithfulness of $\iota \colon \overline{\mp} \hookrightarrow \mp$. Thus, the function
\[ \mE_I\big( (I,f,\alpha), (I,g,\beta) \big) \to \mE'_I\big( \sum_f \alpha, \sum_f \beta\big)\]
induced by $F|_{\mE_I}$ is bijectice, \ie, $F|_{\mE_I}$ is fully faithful.
\end{proof}

An analogous statement can be proven for the $\Pi_\cF$-completion by duality.

Finally, we can combine all of these results to prove our envisioned characterization of Dialectica fibrations as exactly the Gödel fibrations.

\begin{theorem}
    A fibration $\fibration{\mE}{\mp}{\mB}$ over a display map category $\pair{\mB}{\cF}$ with $\cF$-products is an instance of a simple product completion if and only if it has enough $(\cF,\prod)$-quantifier-free objects.
\end{theorem}

\begin{proof}
    This follows from~\Cref{prop_product_via_coproduct_and_op} combined with~\Cref{thm:coprod-compl-enough-quant}.
\end{proof}

Combining the last two results yields the following main result, characterising the dependent Gödel fibrations, up to fibred isomorphism, as the dependent Dialectica fibrations, with respect to a fixed class of display maps:

\begin{theorem}
    Let $\fibration{\mE}{\mp}{\mB}$ be a fibration with $\cF$-products, $\cF$-coproducts and such that $\mB$ has $\cF$-dependent products. Then there exists a fibration $\mp'$ such that $\mathfrak{Dial}_{\cF}(\mp') \cong \mp$ if and only if $\mp$ is a Gödel fibration.

    In particular, $\mp'$ can be taken to be the full subfibration of $(\cF,\prod)$-quantifier-free elements of $\mp$.
\end{theorem}

\section{Conclusions}
G\"odel's Dialectica Interpretation has had many categorical conceptualizations. Philip Scott introduced a completely syntactic version~\cite{Scott71978}. de Paiva \cite{depaiva1991dialectica} introduced a categorification of the construction, by assigning to (a finitely complete) category $\mathsf C$ its \emph{Dialectica category} $\mathsf{Dial}(\mathsf C)$. 
Work of Hyland, Biering, Hofstra, von Glehn, and Moss, generalised the Dialectica construction, assigning to a Grothendieck fibration $\mathsf p: \mathsf E \to \mathsf B$ 
its \emph{Dialectica fibration} $\mathfrak{Dial}(\mathsf p)$. In particular, Hofstra proved that the Dialectica fibration can be obtained as the composition of two free constructions: one adding  (simple) products to a given fibration, and the second adding (simple) coproducts.
Building on Hofstra's work, Trotta et al. proved an \emph{internal} characterisation of the Dialectica construction, introducing  \emph{Skolem and Gödel fibrations}, through the key notion of \emph{quantifier-free elements} of a fibration. 

In this work we extend the previous results to the ``dependent'' case, by replacing the completion process of adding products, then coproducts iteratively with respect to cartesian projections, by adding \emph{dependent} products and coproducts, with respect to a class of display maps $\cal F$.

 Thus the (simple) Dialectica fibration of a fibration $\mathsf p$ gets replaced by its generalised variant $\mathfrak{Dial}_\mathcal F(\mathsf p)$, which arises by freely adding \emph{fibred} products and coproducts along the display maps of $\mathcal F$.  
We also introduce a new class of fibrations, which provides a categorification of the calculus of Hilbert  ($\epsilon$- and $\tau$-) operators. Then, we show that every Hilbert ($\epsilon$- and $\tau$-)fibration is a particular (idempotent) case of a Gödel fibration. 

From an algebraic perspective, as a result of our previous analysis, we obtain that the Hilbert, (locally) Skolem and Gödel fibrations correspond to (suitable) algebras for pseudomonads, unifying the previously unconnected proof-theoretical constructions. These fibrations recover various relevant examples in categorical logic, including the category of polynomials/containers and \textit{a fortiori} (some kinds of) lenses.

For future work, we intend to thoroughly study under which conditions the local isomorphisms considered in this work, i.e., Skolemisation and prenex normal form, extend to global isomorphisms. As previously mentioned, the main idea is to require the existence of a canonical representative and to impose coherence conditions on the representatives.

Finally, we plan to formalize our results in a proof assistant. We estimate that a well-suited framework is given by Hazratpour's recent formalization of fibred categories~\cite{HazratpourLeanFibered} in Lean 4, which would in particular allow for integration into the Mathlib library~\cite{mathlib}.

\paragraph{Related work}
This paper builds on work of \cite{hofstra2011} and \cite{trotta_et_al:LIPIcs.MFCS.2021.87}. Both consider a fibrational view of G\"odel's  Dialectica Interpretation.
The work in \cite{Hyland2002}, \cite{Biering_dialecticainterpretations}, and \cite{mossvonglehn2018} generalise the Dialectica construction, assigning to a Grothendieck fibration $\mathsf p: \mathsf E \to \mathsf B$ (over a finitely complete category $\mathsf B$) its \emph{Dialectica fibration} $\mathfrak{Dial}(\mathsf p): \mathfrak{Dial}(\mathsf E) \to \mathsf B$. The original dependent Dialectica category $\mathsf{Dial}(\mathsf C)$ is recovered (see~\cite{moss2022:polyTalk}) as the fibre over the terminal object $\mathfrak{Dial}(\mathsf{Sum}(!_\mathsf C))_1$ of the Dialectica construction applied to the functor $\mathsf{Sum}(!_\mathsf C) : \mathsf{Sum}(\mathsf C) \to \mathsf{Sum}(1) \simeq \set$, where $\mathsf{Sum}(\cdot)$ denotes the $\set$-indexed free sum completion of a category. In fact, $\mathfrak{Dial}(\mathsf p)$ turning out to be fibred equivalent to the iterated completion of the fibration $\mathsf p$ by first adding fibred products and then fibred sums,  suggests a close connection to von Glehn's~\emph{polynomials}~\cite{vonGlehnPhD}.

\paragraph{Acknowledgments}
For fruitful discussions and helpful feedback we would like to thank Mathieu Anel, Carlo Angiuli, Steve Awodey, Tim Campion, Matteo Capucci, Jonas Frey, Bruno Gavranovi\'{c}, Sina Hazratpour, Milly Maetti, Abdullah Malik, Paul-Andr\'{e} Melli\`{e}s, David Jaz Myers, Emily Riehl, Francisco Rios, Thomas Streicher, and Andrew Swan. This material is based upon work supported by the National Science Foundation under Grant Numbers DMS 1641020 and DMS 1916439, through the American Mathematical Society's Mathematics Research Community (AMS MRC) on applied category theory held in 2022. Any opinions, findings, and conclusions or recommendations expressed in this material are those of the authors and do not necessarily reflect the views of the National Science Foundation. All authors are also grateful to the Hausdorff Research Institute for Mathematics in Bonn, Germany, for hosting us as part of the trimester ``Prospects of formal mathematics,'' funded by the Deutsche Forschungsgemeinschaft (DFG, German Research Foundation) under Germany's Excellence Strategy – EXC-2047/1 – 390685813. Jonathan Weinberger is grateful for financial support by the US Army Research Office under MURI Grant W911NF-20-1-0082.

\bibliographystyle{apalike}
\bibliography{references}

\end{document}